\documentclass[nolineno,a4paper,USenglish,cleveref,thm-restate]{socg-lipics-v2021}
\pdfoutput=1
\usepackage{amsmath,amssymb}
\usepackage{cleveref}
\usepackage{csquotes}
\usepackage{xspace}
\usepackage{hyperref}
\usepackage{graphicx}

\newtheorem{question}{Question}
\Crefname{invariant}{Invariant}{Invariants}
\Crefname{constraint}{Constraint}{Constraints}
\creflabelformat{constraint}{(#2C#1#3)}
\crefrangelabelformat{constraint}{(#3C#1#4--#5C#2#6)}
\Crefname{perturbation}{Condition}{Conditions}
\creflabelformat{perturbation}{(#2#1#3)}
\crefrangelabelformat{perturbation}{(#3#1#4--#5#2#6)}

\newcommand{\lightdfn}[1]{\emph{#1}}
\newcommand{\dfn}[1]{\emph{#1}}

\newcommand{\TODO}[2][\empty]{{\normalfont\color{red}(\ifx #1\empty TODO:\else TODO #1:\fi\ {#2})}}

\newcommand{\bt}{\operatorname{\mathsf{bt}}}
\newcommand{\arb}{\operatorname{\mathsf{a}}}

\newcommand{\height}[1]{\operatorname{\mathrm{height}}(#1)}

\newcommand{\ycoord}[1]{\operatorname{\mathsf{y}}(#1)}
\newcommand{\xcoord}[1]{\operatorname{\mathsf{x}}(#1)}

\newcommand{\Ch}{\ensuremath{\mathsf{h}}}
\newcommand{\Chs}{\ensuremath{\mathsf{hs}}}
\newcommand{\Cv}{\ensuremath{\mathsf{v}}}
\newcommand{\Cvs}{\ensuremath{\mathsf{vs}}}

\newcommand{\typel}[1]{{\normalfont(#1-\textsc{left})}}
\newcommand{\typer}[1]{{\normalfont(#1-\textsc{right})}}
\newcommand{\typei}[1]{{\normalfont(#1-\textsc{int})}}
\newcommand{\typee}[1]{{\normalfont(#1-\textsc{ext})}}

\newcommand{\np}{{\normalfont\textsf{NP}}}

\graphicspath{{figs/}}

\title{On the geometric thickness of 2-degenerate graphs}
\author{Rahul Jain}{FernUniversität in Hagen, Germany}{rahul.jain@fernuni-hagen.de}{https://orcid.org/0000-0002-8567-9475}{}             
\author{Marco Ricci}{FernUniversität in Hagen, Germany}{marco.ricci@fernuni-hagen.de}{https://orcid.org/0000-0002-4502-8571}{}           
\author{Jonathan Rollin}{FernUniversität in Hagen, Germany}{jonathan.rollin@fernuni-hagen.de}{https://orcid.org/0000-0002-6769-7098}{}   
\author{André Schulz}{FernUniversität in Hagen, Germany}{andre.schulz@fernuni-hagen.de}{https://orcid.org/0000-0002-2134-4852}{}         
\authorrunning{R. Jain, M. Ricci, J. Rollin, and A. Schulz}
\Copyright{R. Jain, M. Ricci, J. Rollin, and A. Schulz}

\relatedversion{Extended abstracts of this paper appear at EuroCG 2023~\cite{JRRS23eurocg} and SoCG 2023~\cite{JRRS23socg}.}

\begin{document}
\hideLIPIcs
\maketitle
\hypersetup{pdfsubject={}}

\begin{abstract}
	A graph is $2$-degenerate if every subgraph contains a vertex of degree at most $2$.
	We show that every $2$-degenerate graph can be drawn with straight lines such that the drawing decomposes into $4$ plane forests.
	Therefore, the geometric arboricity, and hence the geometric thickness, of $2$-degenerate graphs is at most $4$.
	On the other hand, we show that there are $2$-degenerate graphs that do not admit any straight-line drawing with a decomposition of the edge set into $2$ plane graphs.
	That is, there are $2$-degenerate graphs with geometric thickness, and hence geometric arboricity, at least $3$.
	This answers two questions posed by Eppstein [Separating thickness from geometric thickness. In \textit{Towards a Theory of Geometric Graphs}, vol.~342 of \textit{Contemp.\ Math.}, AMS, 2004].
\end{abstract}

\section{Introduction}

A graph is planar if it can be drawn without crossings on a plane.
Planar graphs exhibit many nice properties, which can be exploited to solve problems for this class more efficiently compared to general graphs.
However, in many situations, graphs cannot be assumed to be planar even if they are sparse.
It is therefore desirable to define graph classes which extend planar graphs.
Several approaches for extending planar graphs have been established over the last years~\cite{GraphDrawingBeyondPlanarity,BeyondPlanarSurvey}.
Often these classes are defined via drawings, for which the types of crossings and/or the number of crossings are restricted.
A natural way to describe how close a graph is to being a planar graph is provided by the graph parameter \lightdfn{thickness}.
The thickness of a graph $G$ is the smallest number $\theta(G)$ such that the edges of $G$ can be partitioned into $\theta(G)$ planar subgraphs of $G$.
Related graph parameters are \lightdfn{geometric thickness} and \lightdfn{book thickness}.
Geometric thickness was introduced by Kainen under the name \lightdfn{real linear thickness}~\cite{Kainen}.
The geometric thickness $\bar\theta(G)$ of a graph $G$ is the smallest number of colors that is needed to find an edge-colored geometric drawing (i.e., one with edges drawn as straight-line segments) of $G$ with no monochromatic crossings.
For the book thickness $\bt(G)$, we additionally require that only geometric drawings with vertices in convex position are considered.


An immediate consequence from the definitions of thickness, geometric thickness and book thickness is that for every graph $G$ we have $\theta(G) \le \bar \theta(G) \le \bt(G)$.
Eppstein shows that the three thickness parameters can be arbitrarily \enquote{separated}.
Specifically, for any number $k$ there exists a graph with geometric thickness 2 and book thickness at least $k$~\cite{EppGeomBookThickness} as well as a graph with thickness 3 and geometric thickness at least $k$~\cite{EppGeomThickness}.
The latter result is particularly notable since any graph of thickness $k$ admits a $k$-edge-colored drawing of $G$ with no monochromatic crossings if edges are not required to be straight lines.
This follows from a result by Pach and Wenger~\cite{pw01}, stating that any planar graph can be drawn without crossings on arbitrary vertex positions with polylines.

Related to the geometric thickness is the \lightdfn{geometric arboricity} $\bar\arb(G)$ of a graph $G$, introduced by Dujmović and Wood~\cite{DujmovicWoodGeometricArboricity}.
It denotes the smallest number of colors that are needed to find an edge-colored geometric drawing of $G$ without monochromatic crossings where every color class is acyclic.
As every such plane forest is a plane graph, we have $\bar\theta(G) \le \bar\arb(G)$.
Moreover, every plane graph can be decomposed into three forests~\cite{SchnyderRealizer}, and therefore $3\bar\theta(G) \ge \bar\arb(G)$.

%



Bounds on the geometric thickness are known for several graph classes.
Due to Dillencourt et al.~\cite{GeomThicknessCompleteGraphs} we have $\frac{n}{5.646} + 0.342 \le \bar\theta(K_n) \le \frac{n}{4}$ for the complete graph $K_n$.
Graphs with bounded degree can have arbitrarily high geometric thickness.
In particular, as shown by Bar\'art et al.~\cite{BoundedDegreeGeomThickness}, there are $d$-regular graphs with $n$ vertices and geometric thickness at least $c\sqrt{d}n^{1/2 - 4/d - \varepsilon}$ for every $\varepsilon>0$ and some constant $c$.
However, due to Duncan et al.~\cite{GeomThicknessLowDegree}, if the maximum degree of a graph is 4, its geometric thickness is at most 2.
For graphs with treewidth $t$, Dujmovi\'c and Wood~\cite{DujmovicWoodGeometricArboricity} showed that the maximum geometric thickness is $\lceil t/2 \rceil$.
Hutchinson et al.~\cite{hutch} showed that graphs with $n$ vertices and geometric thickness 2 can have at most $6n-18$ edges.
As shown by Durocher et al.~\cite{hardness}, there are $n$-vertex graphs for any $n\ge 9$ with geometric thickness 2 and $2n-19$ edges.
In the same paper, it is proven that it is \np-hard to determine if the geometric thickness of a given graph is at most~2.
Computing thickness~\cite{mansfield} and book thickness~\cite{bookhard} are also known to be \np-hard problems.
For bounds on the thickness for several graph classes, we refer to the survey of Mutzel et al.~\cite{ThicknessSurvey}.
A good overview on bounds for book thickness can be found on the webpage of Pupyrev~\cite{pup}.


A graph $G$ is \lightdfn{$d$-degenerate} if every subgraph contains a vertex of degree at most $d$.
So we can repeatedly find a vertex of degree at most~$d$ and remove it, until no vertices remain.
The reversal of this vertex order (known as a \lightdfn{degeneracy order}) yields a construction sequence for $G$ that adds vertex by vertex and each new vertex is connected to at most $d$ previously added vertices (called its \lightdfn{predecessors}).
Adding a vertex with exactly two predecessors is also known as a Henneberg~1 step~\cite{HennebergGrafischeStatikKoerper}.
In particular, any $2$-degenerate graph is a subgraph of a Laman graph, however not every Laman graph is 2-degenerate.
Laman graphs are the generically minimal rigid graphs and they are exactly those graphs constructable from a single edge by some sequence of Henneberg~1 and Henneberg~2 steps (the latter step consists of subdividing an arbitrary existing edge and adding a new edge  between the subdivision vertex and an arbitrary, yet non-adjacent vertex).
All $d$-degenerate graphs are $(d,\ell)$-sparse, for any $\binom{d+1}{2} \ge \ell \ge 0$, that is, every subgraph on $n$ vertices has at most $dn-\ell$ edges.


\subparagraph*{Our Results.}
In this paper, we study the geometric thickness of 2-degenerate graphs.
Due to the Nash-Williams theorem~\cite{NW61,NW64}, every 2-degenerate graph can be decomposed into 2 forests and hence has arboricity at most~2 and therefore thickness at most~2.
On the other hand, as observed by Eppstein~\cite{EppGeomBookThickness}, 2-degenerate graphs can have unbounded book thickness.
Eppstein's examples of graphs with thickness 3 and arbitrarily high geometric thickness are 3-degenerate graphs~\cite{EppGeomThickness}.
Eppstein asks whether the geometric thickness of 2-degenerate graphs is bounded by a constant from above and whether there are 2-degenerate graphs with geometric thickness greater than $2$.
The currently best upper bound of $O(\log n )$ follows from a result by Duncan for graphs with arboricity~2~\cite{DuncanThickness}.
We improve this bound and answer both of Eppstein's questions with the following two theorems.

\begin{theorem}\label{thm:upperBound}
	For each $2$-degenerate graph $G$ we have $\bar\theta(G) \leq \bar\arb(G) \leq 4$.
\end{theorem}


\begin{theorem}\label{thm:lowerBound}
	There is a $2$-degenerate graph $G$ with $\bar\arb(G) \geq \bar\theta(G) \geq 3$.
\end{theorem}



\section{Proof of \texorpdfstring{\cref{thm:upperBound}}{Theorem~1}: The upper bound}

In this section, we prove \cref{thm:upperBound}.
To this end, we describe, for any $2$-degenerate graph, a construction for a straight-line drawing such that the edges can be colored using four colors, avoiding monochromatic crossings and monochromatic cycles.
This shows that $2$-degenerate graphs have geometric arboricity, and hence geometric thickness, at most four.

Before we give a high-level description of the construction we introduce some definitions.
For a graph $G$ we denote its edge set with $E(G)$ and its vertex set with $V(G)$.
Consider a $2$-degenerate graph $G$ with a given, fixed degeneracy order.
We define the \dfn{height} of a vertex
 $\height{v}$ as the length $t$ of a longest path $u_0\cdots u_t$ with $u_t=v$ such that for each $1 \le i \le t$ vertex $u_{i-1}$ is a predecessor of $u_{i}$.
The set of vertices of the same height is called a \dfn{level} of $G$.
By definition, each vertex has at most two neighbors of smaller height.

Our construction process embeds $G$ level by level with increasing height.
The levels are placed alternately either strictly below or strictly to the right of the already embedded part of the graph.
If a level is placed below, then we use specific colors $\Cv$ and $\Cvs$ (short for \enquote{vertical} and \enquote{vertical slanted}, respectively) for all edges between this level and levels of smaller height.
Similarly, we use specific colors $\Ch$ and $\Chs$ (short for \enquote{horizontal} and \enquote{horizontal slanted}, respectively) if a level is placed to the right.
See \cref{fig:feasibleDrawing} (right).

To make our construction work, we need several additional constraints to be satisfied in each step which we will describe next.
For a point $p$ in the plane, we use the notation $\xcoord{p}$ and $\ycoord{p}$ to refer to the x- and y-coordinates of $p$, respectively.
Consider a drawing $D$ of a $2$-degenerate graph $G$ of height $k$ together with a coloring of the edges with colors $\{\Ch,\Chs,\Cv,\Cvs\}$.
For the remaining proof, we assume that each vertex of $G$ has either $0$ or exactly $2$ predecessors.
If not, we add a dummy vertex without predecessors to the graph and make it the second predecessor of all those vertices which originally only had $1$ predecessor.
We say that $D$ is \dfn{feasible} if it satisfies the following constraints:
\begin{enumerate}[(C1)]
	\item
		\label[constraint]{constr:lastColors}
		For each vertex in $G$ the edges to its predecessors are colored differently.
		If $k>0$, then each vertex of height $k$ in $G$ is incident to one edge of color $\Ch$ and one edge of color $\Chs$.
	\item
		\label[constraint]{constr:separate}
		There exists some $x_D\in\mathbb{R}$ such that for each vertex $v\in V(G)$ we have $\xcoord{v}>x_D$ if and only if $\height{v}=k$.
	\item
		\label[constraint]{constr:planarity}
		There is no monochromatic crossing.
	\item
		\label[constraint]{constr:NAA}
		No two vertices of $G$ lie on the same horizontal or vertical line.
	\item
		\label[constraint]{constr:hopen}
		Each $v\in V(G)$ is \dfn{$\Ch$-open to the right}, that is, the horizontal ray emanating at $v$ directed to the right avoids all $\Ch$-edges.
	\item
		\label[constraint]{constr:vopen}
		Each $v\in V(G)$ is \dfn{$\Cv$-open to the bottom}, that is, the vertical ray emanating at $v$ directed downwards avoids all $\Cv$-edges.
\end{enumerate}
These \lcnamecrefs{constr:planarity} are schematized in \cref{fig:feasibleDrawing}.

\begin{figure}
	\centering
	\includegraphics{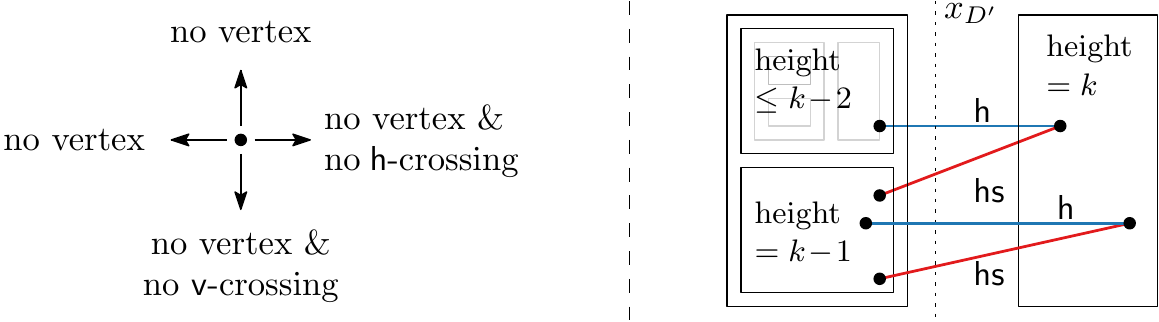}
	\caption{Left: For each vertex $v$ in a feasible drawing, there are no other vertices on the vertical and the horizontal line through $v$.
	Moreover, $v$ is $\Ch$-open to the right and $\Cv$-open to the bottom.
	Right: All vertices in the highest level (of height $k$) are placed to the right of all vertices of smaller height.
	Moreover, each vertex in that level is incident to one edge of color $\Ch$ and one edge of color $\Chs$.}
	\label{fig:feasibleDrawing}
\end{figure}

We now show how to construct a feasible drawing for $G$.
We prove this using induction on the height of the graph.
The base case $k=0$ is trivial, as there are no edges in the graph.
Assume that $k\geq 1$ and the theorem is true for all $2$-degenerate graphs with height $k-1$.
Let $H$ denote the subgraph of $G$ induced by vertices with height less than $k$.
By induction, there is a feasible drawing $D$ of $H$.

As a first step, we reflect the drawing $D$ at the straight line $x=-y$. 
Thus, a point $(x, y)$ before transformation becomes $(-y, -x)$. 
Additionally, we swap the colors $\Chs$ and $\Cvs$ as well as the colors $\Ch$ and $\Cv$.
Let $D'$ denote the resulting drawing.
From now on, all appearing coordinates of vertices refer to coordinates in $D'$.
By construction, $D'$ satisfies \labelcref{constr:planarity,constr:NAA,constr:hopen,constr:vopen}.
Applying \labelcref{constr:lastColors} to $D$ shows that in $D'$ each vertex of height $k-1$ is incident to one edge of color $\Cv$ and one edge of color $\Cvs$.
Applying \labelcref{constr:separate} to $D$ shows that there exists $y_{D'}\in\mathbb{R}$ such that for each vertex $v\in V(H)$ we have $\ycoord{x}<y_{D'}$ in $D'$ if and only if $\height{v}=k-1$.

As the second (and last) step, we place the points of height $k$ of $G$ such that the resulting drawing is feasible.
Let $L_k$ denote the set of these vertices and let $x_{D'}$ denote the largest x-coordinate among all vertices in $D'$.
Choose a sufficiently small slope $m$, with $m>0$, and a sufficiently small $\varepsilon$, with $\varepsilon>0$, such that the following holds.
\begin{enumerate}[(i)]
	\item
		\label[perturbation]{perturb:xR}
		For any distinct $u$, $v \in V(H)$ with $\ycoord{u} < \ycoord{v}$, the horizontal line through $v$ and the straight line through $u$ with slope $m$ intersect at a point $p$ with $\xcoord{p} > x_{D'}$.
	\item
		\label[perturbation]{perturb:NAA}
		For any distinct $u$, $v\in V(H)$ we have that $\varepsilon< \lvert \ycoord{u}-\ycoord{v}\rvert$.
	\item
		\label[perturbation]{perturb:NAAx}
		For any distinct $u$, $u'$, $v$, $v'\in V(H)$ let $p$ be the intersection point of the straight line through $u$ with slope $m$ and the horizontal line through $v$ and let $p'$ be the intersection point of the straight line through $u'$ with slope $m$ and the horizontal line through $v'$.
		If $\xcoord{p}\neq\xcoord{p'}$, then $\varepsilon < \lvert\xcoord{p}-\xcoord{p'}\rvert$.
	\item
		\label[perturbation]{perturb:slantedCrosings}
		For any distinct $u$, $v\in V(H)$ we have that $\varepsilon$ is smaller than the distance between the two straight lines of slope $m$ through $u$ and $v$, respectively.
\end{enumerate}

\begin{figure}
	\centering
	\includegraphics{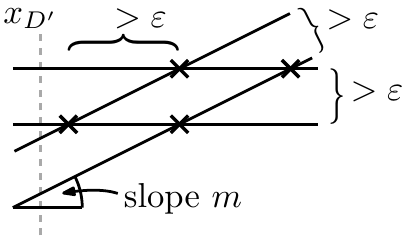}
	\caption{Horizontal lines intersecting straight lines of slope $m$.
	\cref{perturb:NAA,perturb:NAAx,perturb:slantedCrosings} are illustrated.}
	\label{fig:perturbation}
\end{figure}

The constraints are summarized in \cref{fig:perturbation}.
Such a choice of $m$ and $\varepsilon$ is possible, by choosing $m$ according to \cref{perturb:xR} first and then $\varepsilon$ according to the \cref{perturb:NAA,perturb:NAAx,perturb:slantedCrosings}.

For each vertex $w \in L_k$ let $u$ and $v$ be the two predecessors of $w$ in $H$ with $\ycoord{u} < \ycoord{v}$ and let $p^w$ denote the intersection point of the straight line of slope $m$ passing through $u$ (called a slanted line) and the horizontal line passing through $v$.
We will place $w$ close to $p^w$ and connect $w$ to $v$ using an edge of color $\Ch$ and we connect $w$ to $u$ using an edge of color $\Chs$.
To determine the exact location of the vertices, we consider the horizontal lines through vertices $v\in V(H)$ from bottom to top (with increasing y-coordinate) and for each such line consider the intersections with slanted lines through vertices $u\in V(H)$ with $\ycoord{u}<\ycoord{v}$ from left to right (with increasing x-coordinate).
Let $p_1,\ldots,p_t$ denote the intersection points in the order just described.
For each intersection point $p_i$ let $\ell_i$ denote the straight line through $p_i$ with slope $-1/m$ (which is negative as $m>0$), that is, $\ell_i$ is perpendicular to straight lines of slope $m$.
Every vertex $w\in L_k$ with $p^w=p_i$ will be placed on $\ell_i$ at a certain distance from $p^w$ (specified later).
Note that there might be multiple points with the same predecessors and hence multiple vertices $w\in L_k$ with $p^w=p_i$.
For each $p_i$ we order all such vertices arbitrarily.
This gives an ordering of all vertices in $L_k$ based on the ordering $p_1,\ldots,p_t$.
If $w$ is the $k^\text{th}$ vertex in this order, $w$ is placed on $\ell_i$ to the bottom-right of $p_i$ at distance $\varepsilon/2^k$ from $p_i$; see Figure~\ref{fig:pointPlacement}.
In this fashion, all vertices in $L_k$ are placed with decreasing distance to their respective intersection point; see Figure~\ref{fig:pointPlacement2}.

\begin{figure}
	\centering
	\includegraphics{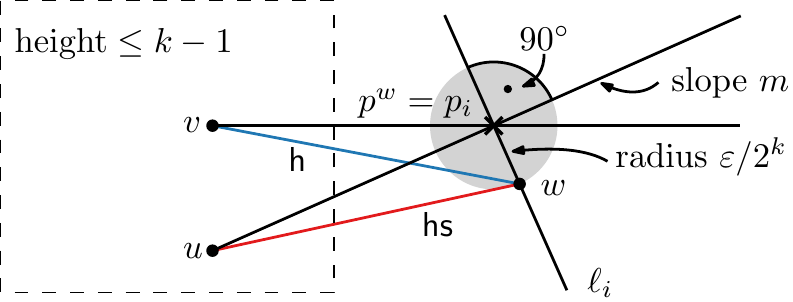}
	\caption{The placement of the $k^\text{th}$ point $w$ in order of vertices in $L_k$.}
	\label{fig:pointPlacement}
\end{figure}

We call the resulting drawing $D_G$.
We claim that $D_G$ demonstrates that the geometric arboricity of $G$ is at most four.

\begin{description}
	\item[{\boldmath Vertices on distinct points, edges intersect in at most one point in $D_G$.}]
		For each $i\le t$ let $\delta_i$ denote the distance between $p_i$ and the first vertex $w$ placed close to $p_i$.
		Then $\delta_i \leq \delta_{i-1}/2$ for each $i$ with $1<i\leq t$.
		For each $i\le t$ let $B_i$ be the region formed by all points $q\in\mathbb{R}$ of distance at most $\delta_i$ to $p_i$ with $\xcoord{q}>\xcoord{p_i}$ and $\ycoord{q}<\ycoord{p_i}$ ($B_i$ is a quarter of a disk).
		Then all vertices $w\in L_k$ with $p^w=p_i$ are placed on distinct points along the intersection of the line $\ell_i$ with $B_i$; see \cref{fig:pointPlacement2}.


		\begin{figure}
			\centering
			\includegraphics{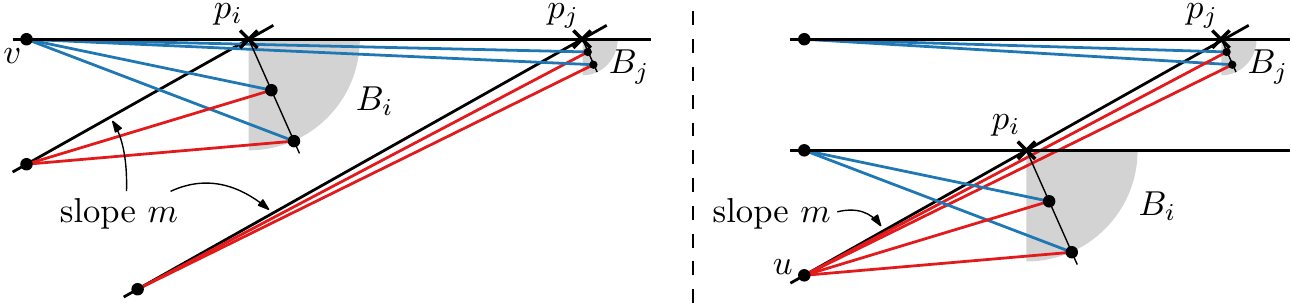}
			\caption{The placement of several points with a common \enquote{horizontal} predecessor $v$ (left) or a common \enquote{slanted} predecessor $u$ (right).
			Edges with color $\Ch$ are drawn blue, edges with color $\Chs$ are drawn red.}
			\label{fig:pointPlacement2}
		\end{figure}

		Due to \cref{perturb:NAA,perturb:slantedCrosings}, all the regions $B_i$ are disjoint.
		By construction, no two vertices are placed on the same point within a region $B_i$.
		This shows that no two vertices in $G$ are placed on the same point in $D_G$.
		Moreover, for the same reasons, for each vertex $v\in V(H)$ the edges between $v$ and vertices in $L_k$ do not contain vertices in their interior and intersect in $v$ only.
		This shows no edge in $G$ contains vertices in its interior and any two edges in $G$ intersect in at most one point.

	\item[\labelcref{constr:lastColors}]
		By construction, each vertex in $L_k$ is incident to an edge of color $\Ch$ and an edge of color $\Chs$.
		Hence, $D_G$ satisfies \labelcref{constr:lastColors}.

	\item[\labelcref{constr:separate}]
		By \cref{perturb:xR}, any horizontal line through some vertex of $H$ and a slanted straight line through a vertex of height $k-1$ in $H$ intersect in some point with x-coordinate larger than $x_{D'}$.
		Each vertex $w\in L_k$ is placed slightly to the right of such an intersection point.
		Hence, $D_G$ satisfies \labelcref{constr:separate}  with $x_{D_G}=x_{D'}$.

	\item[\labelcref{constr:planarity}]
		The edges in the drawing $D'$ of $H$ were not changed, so there are still no monochromatic crossings of those edges.
		Consider an edge $vw$ with $v\in V(H)$ and $w\in L_k$.

		First, assume that its color is $\Ch$.
		Then $\xcoord{w}>\xcoord{v}$ and $\ycoord{w}<\ycoord{v}$ by construction.
		Consider an edge $e$ of color $\Ch$ in $H$.
		We shall prove that $e$ does not cross $vw$.
		If both endpoints of $e$ lie above $v$, then $e$ does not cross $vw$.
		If $e$ crosses the horizontal line through $v$ in some point $p$, then $\xcoord{p}<\xcoord{v}$ since $v$ is $\Ch$-open to the right in $D'$.
		Moreover, one endpoint of  $e$ lies above $v$ while the other endpoint lies below $w$ due to \cref{perturb:NAA}.
		So $e$ does not cross $vw$.
		If both endpoints of $e$ lie below $v$, then their y-coordinates are smaller than $\ycoord{w}$ due to \cref{perturb:NAA}.
		Hence, $e$ does not cross $vw$ in either case.

		Now consider an edge $v'w'$ of color $\Ch$ with $v'\in V(H)$, $\ycoord{v'}<\ycoord{v}$ and $w'\in L_k$.
		As $\ycoord{w}>\ycoord{v'}$ by \cref{perturb:NAA} and $\ycoord{w'}<\ycoord{v'}$ by construction, these two edges do not cross.
		This shows that edges of color $\Ch$ do not cross in $D_G$.

		\begin{figure}
			\centering
			\includegraphics{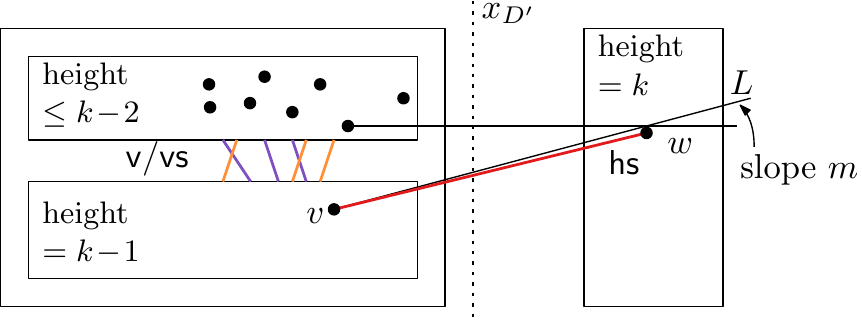}
			\caption{Checking~\cref{constr:planarity} for $\Chs$-colored edges.}
			\label{fig:c3}
		\end{figure}

		Now assume that the color of $vw$ is $\Chs$.
		By construction, $v$ is the predecessor of $w$ of the smallest y-coordinate.
		Since $w$ has at least one predecessor of height $k-1$ and, by induction, all vertices of this height are placed below the vertices of smaller height in $D'$, we have that $\height{v}=k-1$.
		Consider the slanted straight line $L$ (of slope $m$) through $v$.
		By \cref{perturb:xR}, $L$ does not intersect the convex hull of vertices of height less than $k-1$ in $D'$; see~\cref{fig:c3}.
		By induction, all vertices of height $k-1$ in $H$ are incident to edges of color $\Cv$ and $\Cvs$ only.
		Hence, $L$ does not intersect any edge of color $\Chs$ in $D'$.
		The edge $vw$ has a positive slope slightly smaller than $L$ and hence does not intersect any edge of color $\Chs$ in $D'$ either.
		It remains to show that $vw$ does not intersect edges $v'w'$ of color $\Chs$ with $v'\in V(H)$, $v'\neq v$, and $w'\in L_k$.
		Consider the slanted straight line $L'$ (of slope $m$) through $v'$.
		Without loss of generality, assume that $L$ is above $L'$ (the case $L=L'$ produces no crossing since then $v=v'$).
		The edge $v'w'$ has a positive slope slightly smaller than $L'$.
		By \cref{perturb:slantedCrosings}, the distance between $L$ and $w$ is smaller than the distance between $L$ and $L'$.
		Hence $vw$ does not cross $v'w'$.

		This shows that edges of color $\Chs$ do not cross in $D_G$ and hence $D_G$ satisfies \labelcref{constr:planarity}.

	\item[\labelcref{constr:NAA}]
		No two vertices from $H$ lie on a common vertical or horizontal line by induction.
		Consider $w\in L_k$ and the region $B_i$ containing $w$.
		Due to~\cref{perturb:NAA} no horizontal line through $B_i$ contains a vertex from $H$.
		Moreover, by~\labelcref{constr:separate} no vertical line through $B_i$ contains a vertex from $H$.
		Note that either two different regions $B_i/B_j$ are separated by a horizontal line
		or $\ycoord{p_i} = \ycoord{p_j}$. In both cases, vertices placed in $B_i/B_j$ cannot have
		the same y-coordinate. This is clear in the former case and
		in the latter it is true since we never
		select the same distance from $p_i/p_j$ when placing the vertices. For the x-coordinates
		we can argue similarly.
		Hence, $D_G$ satisfies \labelcref{constr:NAA}.

	\item[\labelcref{constr:hopen}]
		First, consider a vertex $v\in V(H)$ and the horizontal ray $L$ emanating at $v$ to the right.
		In the drawing $D'$, each vertex in $H$ is $\Ch$-open to the right, so $L$ does not intersect any $\Ch$-colored edge from $H$.
		It remains to consider $\Ch$-colored edges $v'w$ with $v'\in V(H)$ and $w\in L_k$.
		Then $\xcoord{w}>\xcoord{v'}$ and $\ycoord{v'}> \ycoord{w} > \ycoord{v'}-\varepsilon$ by construction.
		So if $\ycoord{v'}<\ycoord{v}$, $L$ does not intersect $v'w$.
		If $\ycoord{v'}>\ycoord{v}$, then observe that $\ycoord{w}>\ycoord{v'}-\varepsilon > \ycoord{v}$ by \cref{perturb:NAA}.
		Hence $L$ does not intersect $v'w$ in either case and $v$ is $\Ch$-open to the right in $G_D$.

		Now consider a vertex $w\in L_k$ and the horizontal ray $L$ emanating at $w$ to the right.
		By \labelcref{constr:separate}, $L$ does not intersect any edge from $H$.
		It remains to consider $\Ch$-colored edges $v'w'$ with $w'\in L_k$.
		Let $v$ be the neighbor of $w$ in $H$ with $vw$ colored $\Ch$.

		If $v'=v$, consider the region $B_i$ containing $w$.
		If $w'$ is in $B_i$, then $w'$ and $w$ lie on the diagonal $\ell_i$ in $B_i$.
		If $w'$ is in $B_j$ with $j<i$, then $w'$ is placed to the left of $w$, and if $w'$ is on $B_j$ with $j>i$, then $w'$ is placed above $w$.
		In either case, $L$ does not intersect $v'w'$.

		Now suppose that $v'\neq v$.
		Assume that $\ycoord{v'}<\ycoord{v}$ then by \cref{perturb:NAA} and by construction $\ycoord{w}>\ycoord{v'}>\ycoord{w'}$.
		If on the other hand $\ycoord{v'}>\ycoord{v}$ then $\ycoord{v'}>\ycoord{w'}>\ycoord{v} > \ycoord{w}$, again by \cref{perturb:NAA} and by construction.
		In both cases, it follows that $L$ does not intersect $v'w'$.

		This shows that each vertex of $G$ is $\Ch$-open to the right in $D_G$.

	\item[\labelcref{constr:vopen}]
		In the drawing $D'$, each vertex in $H$ is $\Cv$-open to the bottom.
		The vertices in $L_k$ are not incident to any edges of color $\Cv$.
		Hence, all vertices of $G$ are $\Cv$-open to the bottom in $D_G$.
		So \labelcref{constr:vopen} is satisfied.

	\item[No monochromatic cycles.]
		\labelcref{constr:lastColors,constr:separate,constr:planarity,constr:NAA,constr:hopen,constr:vopen} are satisfied, thus $D_G$ is feasible, and uses 4 colors.
		Consider any cycle in $G$ and a vertex $w$ of largest height in the cycle.
		Then its neighbors $u$ and $v$ in the cycle have to be its predecessors.
		Due to \labelcref{constr:lastColors}, $uw$ and $vw$ do not have the same color.
		Hence there are no monochromatic cycles.
\end{description}

\clearpage


\section{Proof of \texorpdfstring{\cref{thm:lowerBound}}{Theorem~2}: The lower bound}\label{sec:lowerBound}

In this section, we shall describe a $2$-degenerate graph with geometric thickness at least~$3$.
For a positive integer $n$ let $G(n)$ denote the graph constructed as follows.
Start with a vertex set $\Lambda_0$ of size $n$ and for each pair of vertices from $\Lambda_0$ add one new vertex adjacent to both vertices from the pair.
Let $\Lambda_1$ denote the set of vertices added in the last step.
For each pair of vertices from $\Lambda_1$ add $89$ new vertices, each adjacent to both vertices from the pair.
Let $\Lambda_2$ denote the set of vertices added in the last step.
For each pair of vertices from $\Lambda_2$ add one new vertex adjacent to both vertices from the pair.
Let $\Lambda_3$ denote the set of vertices added in the last step.
This concludes the construction.
Observe that for each $i=1,2,3$, each vertex in $\Lambda_i$ has exactly two neighbors in $\Lambda_{i-1}$.
Hence, $G(n)$ is $2$-degenerate.
We claim that for sufficiently large $n$ the graph $G(n)$ has geometric thickness at least~$3$.
To prove this result, we need several geometric and topological insights that we outline next. 

\begin{figure}
	\centering
	\includegraphics{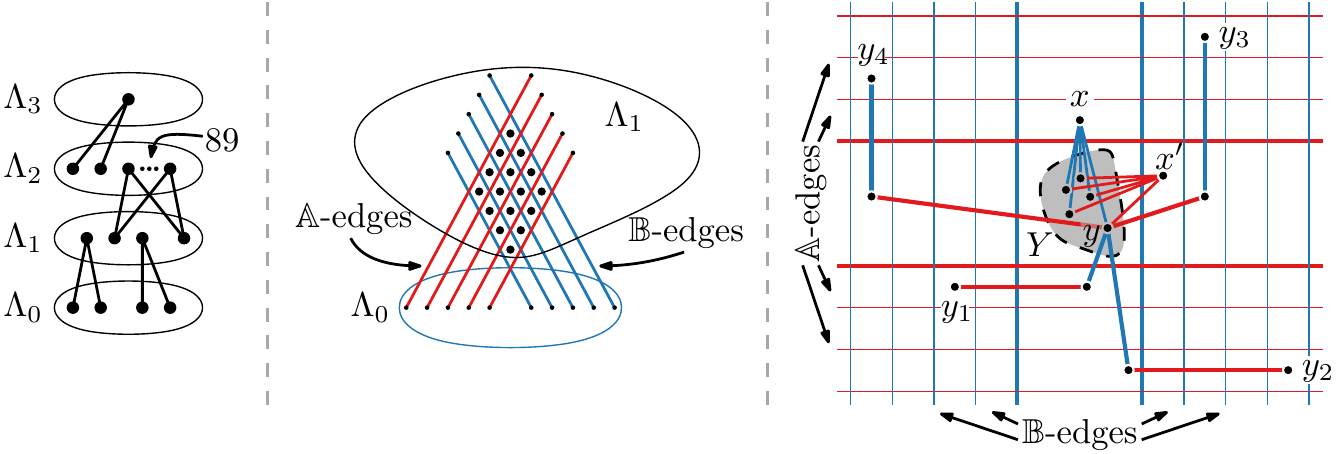}
	\caption{Left: Sketch of the graph $G(n)$. Middle: A tidy grid. Right: The situation leading to a contradiction in the proof of \cref{thm:lowerBound} with $x$, $x'\in\Lambda_1$, $Y\subseteq \Lambda_2$, and $y_1$, $y_2$, $y_3$, $y_4\in\Lambda_3$.}
	\label{fig:lowerBoundSummary}
\end{figure}

We consider a geometric drawing of $G(n)$, for large $n$, and assume that there is a partition of its edge set into two plane subgraphs $\mathbb{A}$ and $\mathbb{B}$.
In the first step, we find a large, and particularly nice grid structure  (called a tidy grid) formed by edges between $\Lambda_0$ and $\Lambda_1$ where many disjoint $\mathbb{A}$-edges cross many disjoint $\mathbb{B}$-edges.
We additionally ensure that there is a large subset $\Lambda'_1\subseteq \Lambda_1$ spread out over many cells of this grid.
Next, we consider the connections of vertices from $\Lambda'_1$ via the edges towards $\Lambda_2$.
We show that the drawing restrictions imposed by the surrounding grid edges force many of the edges between $\Lambda'_1$ and $\Lambda_2$ to stay within the grid.
This gives a large subset $\Lambda'_2\subseteq \Lambda_2$ spread out over many cells of the grid.
Similarly to the previous argument, we then find many of the edges between $\Lambda'_2$ and $\Lambda_3$ staying within the grid.
We eventually arrive at a situation depicted in \cref{fig:lowerBoundSummary} (right):
A cell with a set $Y$ of five vertices from $\Lambda_2$ with the same predecessors in $\Lambda_1$, such that for each $y\in Y$ there are four vertices $y_1,\ldots,y_4\in \Lambda_2$ (one from the bottom-left, one from the bottom-right, one from the top-right, and one from the top-left part of the grid) and for each $i$ the common neighbor of $y$ and $y_i$ from $\Lambda_3$ lies in the grid.
It turns out, that each $y\in Y$ either has an $\mathbb{A}$-edge to the left and an $\mathbb{A}$-edge to the right or it has  a  $\mathbb{B}$-edge to the top and a  $\mathbb{B}$-edge to the bottom (using directions from \cref{fig:lowerBoundSummary}).
As this is impossible to realize for all five vertices in $Y$ simultaneously, the geometric thickness of $G(n)$ is at least~$3$.
We give the full argumentation next.

\subsection[Finding k-grids]{Finding $k$-grids}

\begin{figure}
	\centering
	\includegraphics{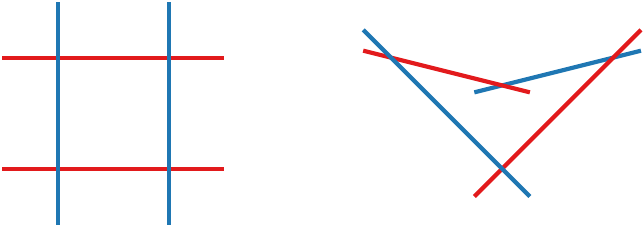}
	\caption{All arrangements of two disjoint red line segments and two disjoint blue segments, such that each red segment crosses each blue segment (up to combinatorial equivalence).}
	\label{fig:2grid}
\end{figure}

We start with observations about the arrangement of line segments.
We call two arrangements of straight lines or straight-line segments \lightdfn{combinatorially equivalent} if the embeddings given by the arrangement of their graphs (skeletons) are combinatorially equivalent.

\begin{lemma}\label{lem:2grid}
	Up to combinatorial equivalence, there are two arrangements of two disjoint red line segments and two disjoint blue segments, such that each red segment crosses each blue segment (see~\cref{fig:2grid}).
\end{lemma}
\begin{proof}
	Consider an arrangement of two disjoint red and two disjoint blue segments where each red segment crosses each blue segment.
	By applying a small perturbation, if necessary, we ensure that the line arrangement obtained by extending each segment to a line is simple.
	\Cref{fig:4StraightLines} (left) shows one such arrangement.
	All simple arrangements of four lines in the plane are combinatorially equivalent.
	There are $\binom{4}{2}=6$ ways to color two lines blue and two lines red.
	\Cref{fig:4StraightLines} shows these possibilities up to permutations of the colors.
	Each coloring induces exactly four red-blue crossings and hence superimposes an arrangement of segments with four red-blue crossings.
	The arrangement satisfies the required properties (for type (I) and (II) in \cref{fig:4StraightLines}) or it violates the disjointness property due to a monochromatic crossing (type (III)).
	This shows that any arrangement of four segments satisfying the required properties is equivalent to an arrangement given in \cref{fig:2grid}.
\end{proof}

\begin{figure}
	\centering
	\includegraphics{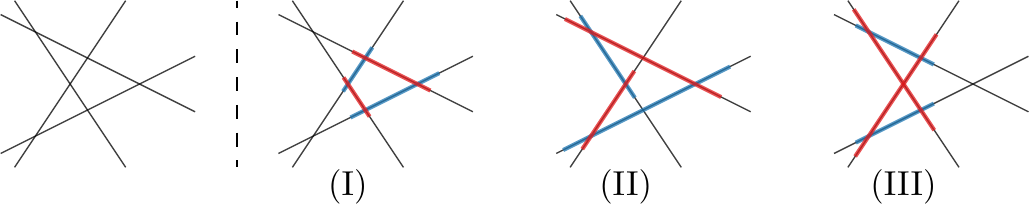}
	\caption{Left: An arrangement of four lines.
	Right: Three possibilities to color the lines red and blue.
	Coloring (I) and (II) each correspond to a combinatorially unique arrangement of colored segments satisfying the requirements of \cref{lem:2grid}.
	Coloring (III) does not correspond to any such arrangement.}
	\label{fig:4StraightLines}
\end{figure}

Let $G_k$ denote the grid formed by $k$ horizontal straight-line segments crossing $k$ vertical straight-line segments.
The grid $G_k$ has four \lightdfn{sides}: the sets of left and right endpoints of the horizontal segments and the sets of lower and upper endpoints of the vertical segments form the four sides of $G_k$, respectively.
The first and the last horizontal segment and the first and the last vertical segment form the \lightdfn{boundary} of $G_k$ while all other segments are called the \lightdfn{inner edges} of $G_k$.
We call an arrangement of straight-line segments combinatorially equivalent to $G_k$ a \lightdfn{$k$-grid}.
We point out that a $k$-grid sometimes refers to a set of disjoint red segments and a set of disjoint blue segments where every pair of red/blue segment intersects; e.g.,~\cite{AFPS14}.
Note that our definition is more restrictive.
Among others, no two segments share an endpoint in our notion of a $k$-grid.
The following lemma shows how both concepts are related.

\begin{restatable}{lemma}{RestateKgrid}\label{lem:kgrid}
	Each arrangement of $k2^{k-1}$ disjoint red straight-line segments and $k$ disjoint blue straight-line segments, where each red segment crosses each blue segment, contains a $k$-grid.
\end{restatable}
\begin{proof}\label{prf:kgrid}
	Let $\mathcal{R}$ and $\mathcal{B}$ denote the sets of red and blue segments, respectively.
	Give every blue segment an arbitrary orientation, then pick one blue segment and name it $\bar{b}$.
	Now orient each red segment such that all red segments cross $\bar{b}$ from left to right.
	Let $b_1,\ldots,b_{k-1}$ denote the $k-1$ blue segments in $\mathcal{B}\setminus\{\bar{b}\}$.
	For each red segment $r$ consider the crossing vector $(\mathrm{bx}_1,\ldots,\mathrm{bx}_{k-1})$ where $\mathrm{bx}_i=0$ if $b_i$ crosses $r$ from left to right and $\mathrm{bx}_i=1$ otherwise.
	Since $\lvert\mathcal{A}\rvert = k2^{k-1}$ there is, by the pigeonhole principle, a set $S$ of $k$ red segments with the same crossing vector.
	We may assume, by reorienting blue segments if necessary, that for each $r\in S$ all blue segments cross $r$ in the same direction as $\bar{b}$ (say, from left to right).

	We claim that $S$ and $\mathcal{B}$ form a $k$-grid.
	To see this, we shall prove that any two red segments from $S$ cross the blue segments from $\mathcal{B}$ in the same order and, similarly, any two blue segments from $\mathcal{B}$ cross the red segments from $S$ in the same order (with respect to the orientations of segments fixed above).
	Consider two segments $r$, $r'\in S$ and two segments $b$, $b'\in\mathcal{B}$.
	We claim that they form an arrangement of type (I).
	Indeed, it is straightforward to check that for each orientation of segments in an arrangement of type (II), the two blue segments cross one red segment from the same side and the other segment from opposite sides.
	As $b$ and $b'$ cross both $r$ and $r'$ from the same side, $r$, $r'$, $b$, and $b'$ form an arrangement of type (I).
	Since all blue segments cross each red segment from left to right, the orientation of segments in an arrangement of type (I) corresponds to one of the orientations given in \cref{fig:orientedCrossings}.
	Thus, the order of crossings is the same along each red and blue segment.
	This shows that $S$ and $\mathcal{B}$ form a $k$-grid.
\end{proof}

\begin{figure}
	\centering
	\includegraphics{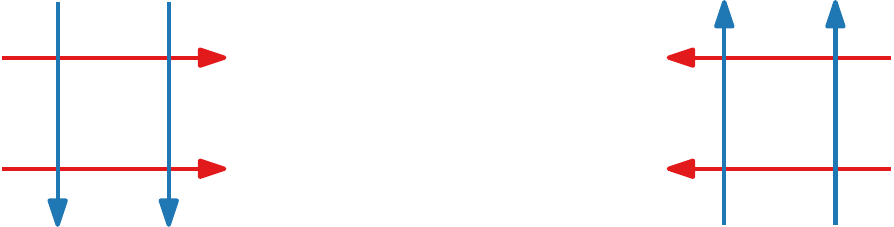}
	\caption{The only possible orientations of segments in an arrangement of type (I) if the blue segments cross each red segment from left to right.}
	\label{fig:orientedCrossings}
\end{figure}

\subsection[Finding tidy k-grids]{Finding tidy $k$-grids}

Observe that a full $1$-subdivision of a graph is $2$-degenerate.
We are particularly interested in subdivisions of large complete bipartite graphs.
Fox and Pach~\cite{FP10} show that there is a constant $c'_1$, with $c'_1\geq 1$, such that for each $k\geq 2$ and for any topological graph on $n$ vertices and more than $n (\log n)^{c'_1\log k}$ edges, there is a set of $k$ independent, pairwise crossing edges.
The following lemma is a direct consequence of this.

\begin{lemma}[{\cite[Lemma 5.3]{FP10}}]\label{lem:pwCrossingEdges}
	There is a constant $c_1$ such that for each $k\geq 3$ and $n\geq 2^{c_1(\log k)^2}$ each topological drawing of $K_{n,n}$ contains a set of $k$ independent, pairwise crossing edges.
\end{lemma}
\begin{proof}
	Let $p=c'_1\log k$ where $c'_1$ is the constant mentioned above.
	We shall show that there is a constant $c_1$ such that with $n\geq 2^{c_1(\log k)^2}$, there are more than $2n(\log (2n))^p$ edges in $K_{n,n}$ (which has $2n$ vertices).
	First observe that the term $(\log(x)-1)/\log\log(2x)$ is increasing and positive for $x>2$.
	Hence, for sufficiently large $c_1$ and $n\geq 2^{c_1(\log k)^2}$ we have
	\[\frac{\log(n)-1}{\log\log (2n)} \geq \frac{c_1(\log k)^2-1}{\log(c_1(\log k)^2+1)} \geq \frac{ \frac{c_1}{2}(\log k)^2}{\log(c_1(\log k)^2+1)} = p\ \frac{\frac{c_1}{2c'_1}\log k}{\log(c_1(\log k)^2+1)} \geq p.\]

	Here the first inequality is based on the lower bound on $n$ while the latter two inequalities hold for sufficiently large $c_1$ (independent of $k$).
	From this we see that $\log(n) \geq \log (2(\log (2n))^p)$.
	Taking powers on both sides shows that the number of edges in $K_{n,n}$ is $n^2 \geq 2n (\log (2n))^p$ as desired.
	Hence, any topological drawing of $K_{n,n}$ contains a set of $k$ independent, pairwise crossing edges.
\end{proof}

\begin{figure}
	\centering
	\includegraphics{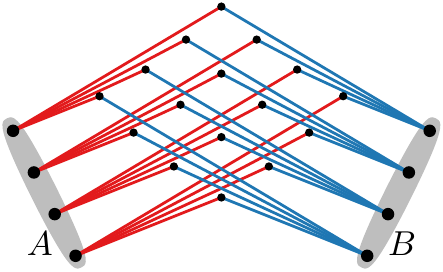}
	\caption{A tidy drawing of $H_4$, the full $1$-subdivision of $K_{4,4}$. In particular, edges incident to $A$ do not cross each other, edges incident to $B$ do not cross each other, and, hence, there are no three pairwise crossing edges.}
	\label{fig:1bendKnn}
\end{figure}

After subdividing the edges there can be no three pairwise crossing edges (see \cref{fig:1bendKnn}) since otherwise we get a monochromatic crossing.
Still, there are two large sets of edges such that each edge from one set crosses each edge from the other set, hence forming a grid as in the definition by Ackermann et al.~\cite{AFPS14}.
To prove this, we use the bipartite Ramsey theorem introduced by Beineke and Schwenk~\cite{BS76}.
Lacking a proper reference for the multicolor version we are using here, we include a standard proof of this result based on the K\H{o}v{\'a}ri--S{\'o}s--Tur{\'a}n theorem \cite{KST54}, which states that every bipartite graph with bipartition classes of size $n$ each and no copy of $K_{k,k}$ contains less than $(k-1)^{1/k}n^{2-1/k}+kn$ edges; see also Irving~\cite{Irving78}.

\begin{lemma}\label{lem:bipartiteRamsey}
	Let $k$ and $r$ be positive integers.
	For each $n\geq (3r)^k$, and each $r$-coloring of $E(K_{n,n})$ there is a copy of $K_{k,k}$ with all edges of the same color.
\end{lemma}
\begin{proof}
	For each $r$-coloring of $E(K_{n,n})$ there is a color class whose number of edges is at least
	\[\frac{n^2}{r} = \frac{n^{1/k}}{r} n^{2-1/k} \geq 3 n^{2-1/k} = 2n^{2-1/k}+n^{2-1/k} \geq (k-1)^{1/k}n^{2-1/k}+kn.\]
	The first inequality holds due to the lower bound on $n$ and the second inequality holds as $(x-1)^{1/x}<2$ and $k \leq n^{1-1/k}$ for $k\leq \log n$.
	Hence, this color class contains a (monochromatic) copy of $K_{k,k}$ due to the K\H{o}v{\'a}ri--S{\'o}s--Tur{\'a}n theorem.
\end{proof}

In the following, we need a grid-structure with some additional properties summarized in the following definitions.
For any point set $Q$ in the plane, we call a straight-line segment in the plane a \lightdfn{$Q$-edge} if it has an endpoint in $Q$.
We call two point sets $A$ and $B$ \lightdfn{separated} if $A\cup B$ is in convex position and the convex hull of $A$ does not intersect the convex hull of $B$ (that is, along the boundary of the convex hull of $A\cup B$ the sets do not interleave).

Consider a complete bipartite graph $K_{n,n}$ with bipartition classes $A$ and $B$.
Let $H_n$ denote the graph obtained from $K_{n,n}$ by subdividing each edge exactly once.
Let $C$ denote the set of subdivision vertices of $H_n$.
Observe that each edge of $H_n$ has one endpoint in $C$ and the other endpoint in $A\cup B$, and hence is either an $A$-edge or a $B$-edge.
We call a geometric drawing of $H_n$ \lightdfn{tidy}, if $A$ and $B$ are separated, there is no crossing between any two $A$-edges, and there is no crossing between any two $B$-edges.
\Cref{fig:1bendKnn} shows a tidy drawing of $H_4$.
Note that we make no (convexity) assumptions on the positions of subdivision vertices.
Since $A$ and $B$ are separated, a tidy drawing induces an ordering of $A$ and $B$ by traversing these points along the convex hull of $A\cup B$ in the counterclockwise direction starting with the vertices in $A$.
An edge of $H_n$ is called an \lightdfn{inner edge} if it is not incident to the first or last vertex of $A$ and not incident to the first or last vertex of $B$ in the order given above.
Similarly, we call an edge of the underlying copy of $K_{n,n}$ an \lightdfn{inner edge} if it corresponds to two inner edges of $H_n$.

Consider a $k$-grid $T$ in $H_n$ with one side in $A$ and one side in $B$ (and the respective opposite sides in $C$).
We call the sides of $T$ that are contained in $A$ or $B$ the $A$-side and $B$-side, respectively.
Let $a_1,\ldots,a_k$ denote the vertices of the $A$-side of $T$ in the order given by $A$ and let $b_1,\ldots,b_k$ denote the vertices of the $B$-side of $T$ in the order given by $B$.
For each $i$ let $x^A_i$ denote the crossing point between the $A$-edge of $T$ with endpoint $a_i$ and the $B$-edge in $T$ farthest away from $a_i$.
For $i$, $j\le k$, with $i<j$, the \lightdfn{$A_{i,j}$-corridor} of $T$ is the polygon enclosed by $x^A_i$, $a_i$, $a_{i+1},\ldots, a_j$, $x^A_j$.
Crossing points $x^B_1,\ldots,x^B_k$ and \lightdfn{$B_{i,j}$-corridors} are defined similarly.
\Cref{fig:untidyGrid} (right) shows examples of such corridors.
A \lightdfn{tidy $k$-grid} is a topological subgraph $T$ of a tidy drawing of $H_n$ such that
\begin{itemize}
	\item $T$ is a $k$-grid with one side in $A$ and one side in $B$ (and the opposite sides in~$C$),
	\item for each $i\le k$, the segment $a_ix^A_i$ is contained in the $A_{1,k}$-corridor of $T$,
	\item for each $i\le k$, the segment $b_ix^B_i$ is contained in the $B_{1,k}$-corridor of $T$.
\end{itemize}
\Cref{fig:untidyGrid} shows a tidy 3-grid and a 4-grid that is not tidy.

\begin{figure}
	\centering
	\includegraphics{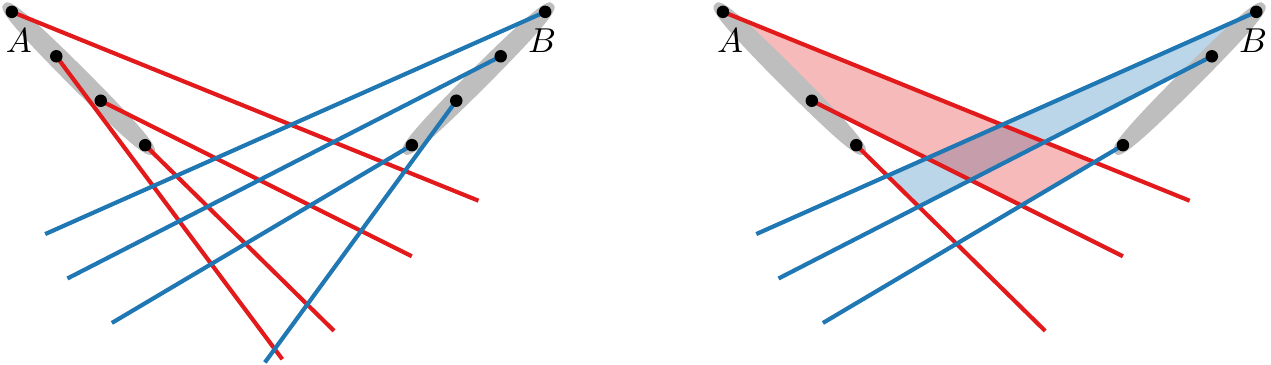}
	\caption{Left: A $4$-grid with sides in $A$ and $B$ that is not tidy: there is a (red) $A$-edge not contained in the $A_{1,k}$-corridor as well as a (blue) $B$-edge not contained in the $B_{1,k}$-corridor.
	Right: A tidy (sub)grid.
	The $A_{1,2}$-corridor and the $B_{2,3}$-corridor are highlighted.}
	\label{fig:untidyGrid}
\end{figure}

Our arguments require a tidy grid such that every cell contains a (subdivision) vertex from $C$.
Such a grid is called \lightdfn{dotted}.
To find a dotted tidy $k$-grid we will show that there is a tidy $k$-grid in any drawing $H_n$ such that between its $A$-side and its $B$-side we have a copy of $H_k$ with all edges staying within the grid.
To achieve this, we shall find a tidy $k$-grid and a copy of $H_k$ pointing to the same \enquote{sides} of $A$ and $B$.

We first need a precise definition of sides.
Consider a point set $Q$ in convex position in the plane (which will be either $A$ or $B$ later on) consisting of points $q_1,\ldots,q_k$, $k\geq 3$, labeled in counterclockwise order around the convex hull of $Q$.
We distinguish four types of $Q$-edges.
Any segment with an endpoint $q_i\in Q$ and the other endpoint not in the convex hull of $Q$ is of type
\begin{description}
	\item[\typel{$Q$}] if it intersects some segment $q_jq_{j+1}$ with $1\leq j < i-1$,
	\item[\typer{$Q$}] if it intersects some segment $q_jq_{j+1}$ with $i< j < k$,
	\item[\typei{$Q$}] if it intersects the segment $q_1q_k$,
	\item[\typee{$Q$}] if it does not intersect the interior of $Q$.
\end{description}

Observe that any $Q$-edge that leaves its endpoint in $Q$ towards the interior of $Q$, and does not have its other endpoint in the convex hull of $Q$, is of one of the first three types.
See \cref{fig:Qtypes} for an illustration.

\begin{figure}
	\centering
	\includegraphics{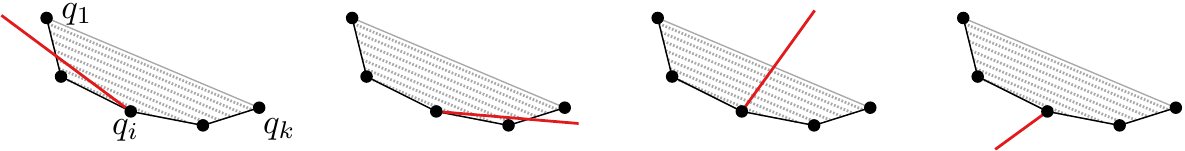}
	\caption{From left to right: types \typel{$Q$}, \typer{$Q$}, \typei{$Q$}, \typee{$Q$}.}
	\label{fig:Qtypes}
\end{figure}

Consider a tidy drawing $D$ of $H_n$ and let $A$ and $B$ denote the bipartition classes of the underlying topological drawing $K$ of $K_{n,n}$.
There are $16$ different types of edges in $K$: each edge in $K$ is of one of the four $A$-types and of one of the four $B$-types.
In a tidy grid, all inner $A$-edges are of the same $A$-type and all inner $B$-edges are of the same $B$-type.
However, only the \typei{$A$/$B$} and \typee{$A$/$B$} types are possible.

We can now show how to find a large copy of $H_k$ with all edges of a suitable kind of type.

\begin{lemma}\label{lem:sameTypeBiclique}
	Let $n$, $k$ be integers with $n\geq 48^{k^2}$.
	For each tidy drawing $D$ of $H_n$, there is a tidy drawing of $H_k$ in $D$ such that for the $A$-side $A'$ of $H_k$ all inner $A'$-edges are of type \typei{$A'$} or all $A'$-edges are of type \typee{$A'$} and for the $B$-side $B'$ of $H_k$ all inner $B'$-edges are of type \typei{$B'$} or all $B'$-edges are of type \typee{$B'$}.
\end{lemma}
\begin{proof}
	Let $k_1=(k-1)^2+1$ and $n\geq 48^{k^2} \geq 48^{k_1}$.
	Consider a tidy drawing $D$ of $H_n$ and let $A$ and $B$ denote the bipartition classes of the underlying topological drawing $K$ of $K_{n,n}$ and let $C$ denote the set of subdivision vertices.
	By \cref{lem:bipartiteRamsey} with $r=16$, there is a copy $K'$ of $K_{k_1,k_1}$ in $K$ with all edges of the same type, that is, all edges of the same $A$-type and all edges of the same $B$-type.
	If this $A$-type is from $\{\text{\typei{$A$}},\text{\typee{$A$}}\}$ and this $B$-type is from $\{\text{\typei{$B$}},\text{\typee{$B$}}\}$, then we are done.
	Indeed, observe that these types are preserved for the inner edges of $K'$ when reducing the set $A\cup B$ to the vertex set of $K'$.

	Otherwise, assume that all edges of $K'$ are of type \typel{$A$}.
	We shall find a subset $A'\subseteq A$ of $k$ vertices such that the inner $A'$-edges in $H_n$ are either all of type \typei{$A'$} or all of type \typee{$A'$}.
	Consider the polygonal chain formed by $A$ in order $a_1,\ldots,a_n$.
	For each vertex $a_i\in A\cap V(K')$, each edge in $K'$ incident to $a_i$ intersects the chain in some segment $a_ja_{j+1}$ with $1\leq j<i-1$ due to its type.
	We call the part of this edge between $a_i$ and the intersection with the chain an \lightdfn{$a_i$-chord} of $A$.
	Let $e_i$ denote the $a_i$-chord whose intersection with the polygonal chain is closest to $a_i$ along the chain.
	Since $D$ is tidy, all the chords with different endpoints $a_i$ do not intersect.
	That is, two chords with different endpoints $a_i$ are either nested or separated.
	See \cref{fig:chords}.

	Since $k_1\geq (k-1)^2+1$ the set $\{e_i\colon a_i\in V(K)\cap A\}$ contains a subset $\Theta$ with $\lvert\Theta\rvert=k$ such that the chords in $\Theta$ are all pairwise nested or all pairwise separated.
	In case the chords are pairwise nested, we choose $A'=\{a_i\colon a_i\in V(K)\cap A, e_i\in\Theta\}$.
	Let $a'$ and $a''$ denote the vertices in $A'$ with the smallest and largest index $i$, respectively.
	Then all inner $A'$-edges in $H_n$ are of type \typei{$A'$}, as they cross the convex hull of $A'$ in the segment $a'a''$.
	See \cref{fig:chords} (left).
	In case the chords are pairwise separated, we choose $A'=\{a_{i-1}\colon a_i\in V(K)\cap A, e_i\in\Theta\}$.
	Observe that $a_{i-1}\in A'$ lies between the two endpoints of the chord $e_i$ and any $A'$-edge in $H_n$ with an endpoint $a_{i-1}\in A'$ does not intersect the chord $e_i$.
	Hence, all $A'$-edges in $H_n$ are of type \typee{$A'$}, as they do not intersect the interior of the convex hull of $A'$.
	See \cref{fig:chords} (right).

	Altogether, we see that we find a set $A'$ of size $k$ such that all $A'$-edges in $H_n$ are of the same type from $\{\text{\typei{$A'$}},\text{\typee{$A'$}}\}$, as desired.

	The other case, when all $A$-edges are of type \typer{$A$}, as well as the $B$-types of edges in $K'$, can be treated with symmetric arguments.
\end{proof}

\begin{figure}
	\centering
	\includegraphics{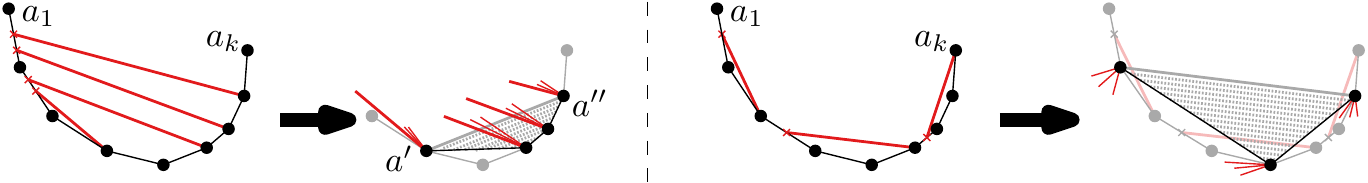}
	\caption{Non-crossing chords of a polygonal chain are either nested (left) or separated (right).}
	\label{fig:chords}
\end{figure}

Finally, we are ready to prove the existence of a dotted tidy grid in any tidy drawing of $H_n$.

\begin{restatable}{lemma}{RestateTidyGrid}\label{lem:tidyGrid}
	There is a constant $c_2$ such that for any integers $n$ and $k$, with $n\geq 2^{c_2 k^4 2^{8k}}$ and $k\geq 3$, each tidy drawing of $H_n$ contains a dotted tidy $k$-grid.
\end{restatable}
\begin{proof}\label{prf:tidyGrid}
	Let $k_1=2k+1$, $k_2=k_1 2^{k_1-1}$, $k_3=6^{k_2}$, and $k_4=2^{c_1(\log(2k_3))^2}$, where $c_1$ is the constant from \cref{lem:pwCrossingEdges}, and choose $c_2$ sufficiently large such that \[2^{c_2k^4 2^{8k}} \geq 2^{\log(48) c_1^2 (1+\log(6) (2k+1) 2^{2k})^4} = 48^{k_4^2}.\]
	Consider $n\geq 2^{c_2 k^4 2^{8k}}$ and a tidy drawing $D''$ of $H_n$.
	Let $A''$ and $B''$ denote the bipartition classes of the underlying topological drawing of $K_{n,n}$ and let $C''$ denote the set of subdivision vertices.
	By \cref{lem:sameTypeBiclique} there is a tidy drawing $D'$ of $H_{k_4}$ in $D''$ such that for its $A''$-side $A'$ all $A'$-edges are of type \typei{$A'$} or all $A'$-edges are of type \typee{$A'$} and for its $B''$-side $B'$ all $B'$-edges are of type \typei{$B'$} or all $B'$-edges are of type \typee{$B'$}.
	Let $K'$ denote the drawing of the underlying $K_{k_4,k_4}$ of $D'$.
	By \cref{lem:pwCrossingEdges}, there is a set $S$ of $2k_3$ independent, pairwise crossing edges in $K'$.
	Partition $S$ arbitrarily into two sets $S_1$ and $S_2$ of size $k_3$ each.
	For each edge $ab$ in $K'$ let $x_{ab}$ denote the corresponding subdivision vertex in $D'$.
	Since the drawing is tidy, there is no crossing of two $A'$-edges and no crossing of two $B'$-edges.
	So each crossing of edges $ab\in S_1$ and $a'b'\in S_2$ in $K'$ corresponds either to a crossing of $ax_{ab}$ and $b'x_{a'b'}$ (A) or a crossing of $a'x_{a'b'}$ and $bx_{ab}$ (B) in $D'$.
	Note that two edges in $K'$ might cross up to two times since we consider a geometric drawing of $H_n$ (so each edge of $K'$ has one bend).
	To find a large set of $A'$-edges crossing a large set of $B'$-edges consider the complete bipartite graph $K_S$ with bipartition classes $S_1$ and $S_2$.
	Each edge in this graph corresponds to a crossing of an edge from $S_1$ with an edge from $S_2$.
	Since $k_3=6^{k_2}$ there are, by \cref{lem:bipartiteRamsey}, sets $S'_1\subseteq S_1$ and $S'_2\subseteq S_2$ of size $k_2$ such that the crossings formed by edges from $S'_1$ with edges from $S'_2$ are all of type (A) or all of type (B).
	In the first case, the $A'$-edges coming from edges in $S'_1$ cross the $B'$-edges coming from $S'_2$.
	In the other case, the $A'$-edges coming from edges in $S'_2$ cross the $B'$-edges coming from $S'_1$.

	We now have a set of $k_2$ $A'$-edges and a set of $k_2$ $B'$-edges such that each $A'$-edge crosses each $B'$-edge.
	Since $k_2=k_1 2^{k_1-1}$ there is, by \cref{lem:kgrid}, a subset $S_A$ of these $A'$-edges and a subset $S_B$ of these $B'$-edges forming a $k_1$-grid $T$.
	This grid has, by construction, one side in $A'$ (and the opposite side in $C''$) and one side in $B'$ (and the opposite side in $C''$).
	Let $A$ and $B$ denote the $A'$-side and $B'$-side of $T$, respectively.
	An example is given in \cref{fig:untidyGrid}.
	It remains to prove that $T$ is tidy and contains a dotted tidy $k$-grid.

	To prove that $T$ is tidy, consider the $A_{1,k_1}$-corridor of $T$.
	Consider the edge $b\in S_B$ that lies on the boundary of this corridor (that is, $b$ is farthest away from $A$) and for each $a\in A$ let $x_a$ denote the crossing point of edge $b$ with the $A$-edge with endpoint $a$.
	Since $T$ is a $k_1$-grid, all points in $A$ lie on the same side of the straight line through $b$.
	In particular, $b$ does not cross the polygonal chain formed by $A$.
	Due to the choice of $K'$, all $A$-edges in $T$ are either of type \typei{$A'$} or all are of type \typee{$A'$}.
	This type is preserved with respect to $A$, except for the $A$-edges incident to the first and last vertex in $A$ which are always type \typee{$A$}.
	In particular, the $A$-edges do not cross the polygonal chain formed by $A$.
	Hence, the $A_{1,k_1}$-corridor of $T$ either contains the convex hull of $A$ entirely or is disjoint from its interior.
	This gives four different possibilities based on the type of $A$-edges in $T$ and the relative positions of the $A_{1,k_1}$-corridor and convex hull of $A$; \cref{fig:sameSideCorridors} shows an illustration.
	If the corridor contains the convex hull of $A$ and $A$-edges are of type \typei{$A'$} (Part (I) of \cref{fig:sameSideCorridors}), then all segments $ax_a$, $a\in A$, are contained in the $A_{1,k_1}$-corridor of $T$ as desired.
	The same holds if the corridor does not contain the convex hull of $A$ and $A$-edges are of type \typee{$A'$} (Part (III) of \cref{fig:sameSideCorridors}).
	The remaining two situations do not occur:
	If the corridor contains the convex hull of $A$ and $A$-edges are of type \typee{$A'$} (Part (II) of \cref{fig:sameSideCorridors}), then both endpoints of $b$ are not in convex position together with $A$, a contradiction to the fact that $H_n$ is tidy and hence $A\cup B$ is in convex position.
	If the corridor does not contain the convex hull of $A$ and $A$-edges are of type \typei{$A'$} (Part (IV) of \cref{fig:sameSideCorridors}), observe that the points in $A$ together with the crossing points of $b$ with the first and last $A$-edge of $T$ are in convex position (since also the first and the last $A$-edge of $T$ are of type \typei{$A'$} in the larger drawing $D'$).
	This shows that the $A_{1,k_1}$-corridor coincides with the convex hull of this set.
	Hence, the $A_{1,k_1}$-corridor contains the convex hull of $A$, contradicting the assumption.
	Altogether, the $A_{1,k_1}$-corridor of $T$ contains all segments $ax_a$, $a\in A$ in each possible case.
	Similar arguments applied to the $B$-side show that $T$ is tidy.

	It remains to show that $T$ contains a dotted tidy $k$-grid.
	To this end consider the drawing $D$ of $H_{k_1}$ in $D'$ between $A$ and $B$.
	By the choice of $D'$, the edges in $T$ and all the edges (of the underlying copy of $K_{k_1,k_1}$) of $D$ are of the same type.
	Let $a_1,\ldots,a_{k_1}$ and $b_1,\ldots,b_{k_1}$ denote the vertices of $A$ and $B$ in counterclockwise order, respectively.
	Let $T'$ denote the drawing formed by all edges in $T$ with an endpoint in $\{a_i\in A\colon i\text{ odd}\}$ or in $\{b_j\in B\colon j\text{ odd}\}$ and all vertices in $D$ that subdivide an edge $a_ib_j$ (of the underlying copy of $K_{k_1,k_1}$) where $1\leq i,j\leq k_1$ and $i$ and $j$ are even.
	See \cref{fig:dottedGrid}.
	By construction, $T'$ is a tidy $k$-grid since $k_1=2k+1$.
	We claim that $T'$ is dotted.
	Consider an arbitrary cell of $T'$.
	It corresponds to the intersection of an $A_{i-1,i+1}$-corridor and a $B_{j-1,j+1}$-corridor of $T$ for some even numbers $i$ and $j$.
	Let $x$ denote the (unique) vertex in $D$ adjacent to both $a_i$ and $b_j$ (the subdivision vertex of the edge $a_ib_j$ in the underlying copy of $K_{k_1,k_1}$).
	We claim that $x$ lies in this cell.
	The $A_{i-1,i+1}$-corridor of $T$ is a polygon with corners $x_{i-1}$, $a_{i-1}$, $a_{i}$, $a_{i+1}$, $x_{i+1}$ (where $x_{i-1}$ and $x_{i+1}$ are crossings in $T'$) where the sides $x_{i-1}a_{i-1}$ and $a_{i+1}x_{i+1}$ are part of $A$-edges.
	Since $A$-edges do not cross each other and $a_ix$ is of the same type as all $A$-edges in $T$, the edge $a_ix$ might leave the corridor only through the side $x_{i-1}x_{i+1}$.
	In particular, it crosses all $B$-corridors of $T'$ in this case.
	Similarly, the edge $b_jx$ cannot leave the $B_{j-1,j+1}$-corridor before it crosses all $A$-corridors of $T$.
	This shows that $x$ is in the desired cell of $T'$, as otherwise $a_ix$ and $b_jx$ cross (in that cell), which is not possible.
	See \cref{fig:dottedGrid} (right).
	Hence $T'$ is a dotted, tidy, $k$-grid.
\end{proof}

\begin{figure}
	\centering
	\includegraphics{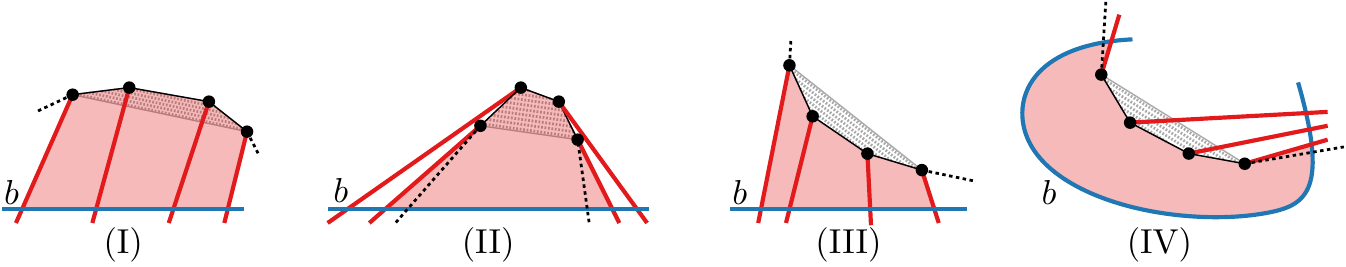}
	\caption{Relative positions of an $A_{1,k}$-corridor (light red area) and the convex hull of $A$ (dotted area): (I) The corridor contains the convex hull of $A$ and $A$-edges are of type \typei{$A'$}.
	(II) The corridor contains the convex hull of $A$ and $A$-edges are of type \typee{$A'$}.
	(III) The corridor does not contain the convex hull of $A$ and $A$-edges are of type \typee{$A'$}.
	(IV) The corridor does not contain the convex hull of $A$ and $A$-edges are of type \typei{$A'$}.}
	\label{fig:sameSideCorridors}
\end{figure}

\begin{figure}
	\centering
	\includegraphics{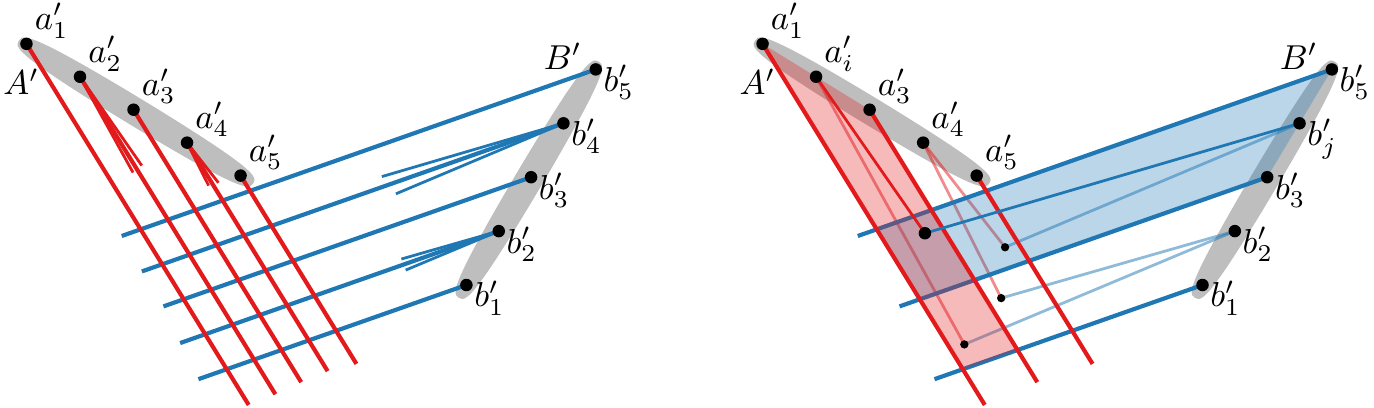}
	\caption{Left: A tidy $5$-grid such that the edges of $K_{5,5}$ between its sides are of the same type as the grid edges (here \typee{$A$} and \typei{$B$}, respectively).
	Right: This grid contains a dotted, tidy $3$-grid formed by edges with odd indices in $A$ and $B$.}
	\label{fig:dottedGrid}
\end{figure}

\subsection{Structures within the grid}

Next, we will consider connections between the vertices inside of a dotted grid.
To find such connections running in certain directions within the grid, we shall use a Ramsey type argument, summarized in the following \cref{lem:gridRamsey}.
We will apply this lemma in such a way that the mentioned color $r$ corresponds to connections within the grid.
For positive integers $k$ and $t$ let $\Gamma(k,t)$ denote the graph whose vertex set consists of disjoint sets $V_i^j$, $i$,$j\le k$, on $t$ vertices each, such that $u\in V_i^j$ and $v\in V_p^q$ are adjacent if and only if $i\neq p$ and $j\neq q$.
See \cref{fig:complementGridGraph} (left) for an illustration of $\Gamma(5,4)$.
Let $r\geq 3$.
We call an $r$-coloring of $E(\Gamma(k,t))$ \lightdfn{admissible} if each monochromatic copy of $K_5$ is of color $r$ and any path $uvw$ is not monochromatic in some color $c$ with $3\leq c<r$ in case $u\in V_i^j$, $v\in V_p^q$, and $w\in V_x^y$ with $1\leq i<p<x\leq k$ and with $1\leq j<q<y\leq k$ or $1\leq y<q<j\leq k$.
Loosely speaking, $\Gamma(k,t)$ is the $t$-blowup of the complement of a $k\times k$-grid graph, and an $r$-coloring is admissible if any monochromatic copy of $K_5$ has color $r$ and each monotone monochromatic path on at least two edges is colored with some color in $\{1,2,r\}$.
Given $i$ and $j$, the \lightdfn{$(i,j)$-quadrants of $\Gamma(k,t)$} are the four subgraphs induced by $\bigcup\limits_{p<i,q<j}V_p^q$, $\bigcup\limits_{p<i,q>j}V_p^q$, $\bigcup\limits_{p>i,q<j}V_p^q$, and $\bigcup\limits_{p>i,q>j}V_p^q$, respectively.
See \cref{fig:complementGridGraph} for an illustration.

\begin{figure}
	\centering
	\includegraphics{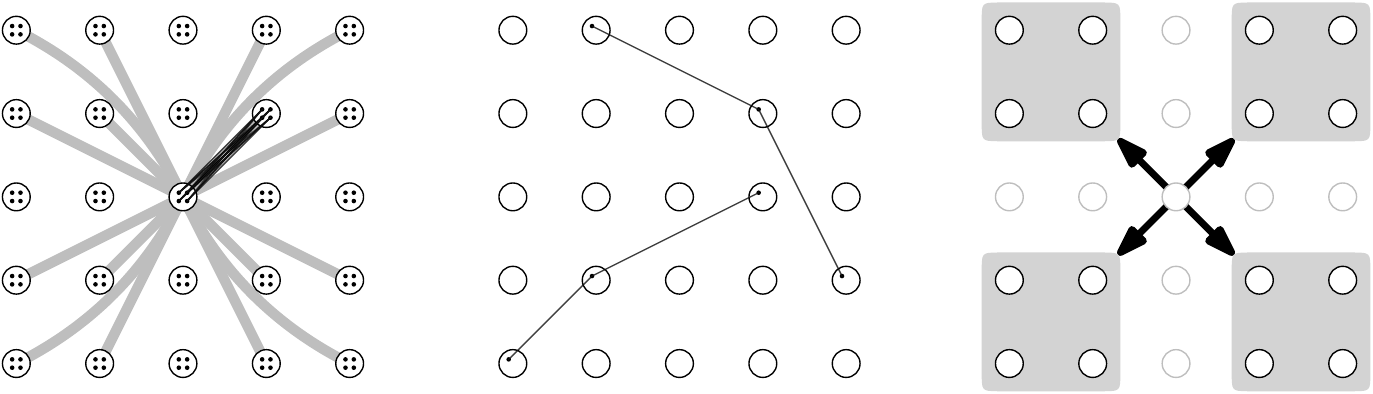}
	\caption{Left: An illustration of $\Gamma(5,4)$.
	Only edges incident to the central vertices are sketched.
	Middle: Two monotone paths in $\Gamma(5,4)$.
	Right: The four $(3,3)$-quadrants.}
	\label{fig:complementGridGraph}
\end{figure}

\begin{restatable}{lemma}{RestateGridRamsey}\label{lem:gridRamsey}
	Let $r$ and $t$ denote positive integers.
	There is a constant $c_3$ such that for each $k\geq c_3$ and each admissible $r$-coloring of $E(\Gamma(k,t))$ there are $i$, $j\le k$ such that each vertex in $V_i^j$ is incident to four edges of color $r$ with endpoints in different $(i,j)$-quadrants.
\end{restatable}
\begin{proof}\label{prf:gridRamsey}
	An ordered graph is a graph together with a fixed linear order $<$ of vertices.
	Let $R$ denote the $(r^{(t^2)})$-color Ramsey number of $K_5$.
	We shall prove that $c_3=R^2$ is sufficient.
	To this end we first prove that for any admissible coloring of $\Gamma(R^2,t)$ every \enquote{diagonal} of length at least $R$ contains a vertex with two edges of color $r$ in two opposite quadrants.
	Then we choose $R$ such vertices from $R$ different diagonals to find a vertex with two additional edges of color $r$ into the other quadrants.

	Consider the $t$-blowup $K_R^t$ of a complete graph on $R$ vertices, that is, the vertex set consists of $R$ disjoint sets $U_1,\ldots,U_R$ of size $t$ each and two vertices are adjacent if and only if they are not in the same set $U_i$.
	Fix an ordering $<$ of the vertices of $K_R^t$ such that for $u\in U_i$ and $v\in U_j$ with $i<j$ we have $u<v$.
	(Within each set $U_i$, $i\le R$, the ordering is arbitrary.)
	Let $\Phi$ denote an $r$-coloring of $E(K_R^t)$ where each monochromatic copy of $K_5$ in $K_R^t$ has color $r$.
	There are $r^{(t^2)}$ possible colorings of the $t^2$ edges between any two vertex sets $U_i$ and $U_{i'}$.
	Thus, we can consider $\Phi$ as an $r^{(t^2)}$-coloring of $E(K_R)$.
	By the choice of $R$, there are integers $j_1,\ldots,j_5\le R$ such that $U_{j_1},\ldots,U_{j_5}$ span a monochromatic $K_5$ in the $r^{(t^2)}$-coloring of $K_R$.
	Hence, for every $\ell\le t$ the $\ell$'th vertices of $U_{j_1},\ldots,U_{j_5}$ induce a monochromatic copy of $K_5$ in $K_R^t$.
	By assumption on $\Phi$, this copy is of color $r$ for each $\ell\le t$.
	In particular, for each vertex $v\in U_{j_2}$ there are vertices $u\in U_{j_1}$, $w\in U_{j_3}$ with $\Phi(uv)=\Phi(vw)=r$.

	Now consider an admissible $r$-coloring $\Phi$ of $E(\Gamma(R^2,t))$.
	For each $d\le R$ let $\Gamma_d$ denote a specific \enquote{diagonal}, namely the copy of $K_R^t$ in $\Gamma(R^2,t)$ induced by vertex sets $V_i^j$ with $i$, $j\in\{(d-1)R+1,\ldots,dR\}$ and $i+j=(2d-1)R+1$.
	See \cref{fig:gridRamsey} (left \& middle).
	Observe that for each set $V_i^j$ in such a diagonal, the other vertices of the diagonal are contained in at most two $(i,j)$-quadrants of $E(\Gamma(R^2,t))$.
	For each $d$, each monochromatic copy of $K_5$ in $\Gamma_d$ under $\Phi$ is of color $r$.
	Hence, as argued above, for each $d$ there is a some set $U_d\in\{V_i^j\colon i, j\in\{(d-1)R+1,\ldots,dR\}, i+j=(2d-1)R+1\}$ in $\Gamma_d$ such that each vertex in $U_d$ is incident to two edges of color $r$ within $\Gamma_d$ with endpoints in opposite quadrants.
	Now consider the \enquote{diagonal} $\Gamma_0$ in $E(\Gamma(R^2,t))$ induced by $U_1,\ldots,U_R$.
	Observe that for each set $V_i^j = U_d$ in this diagonal there are exactly two opposing $(i,j)$-quadrants containing the remaining vertices of $\Gamma_0$ and these quadrants are different from the quadrants considered for the $r$-colored edges incident to $V_i^j = U_d$ within $\Gamma_d$.
	Again, the arguments from above yield a set $U\in\{U_1,\ldots,U_R\}$ such that each vertex in $U$ is incident to two edges of color $r$ within $\Gamma_0$ with endpoints in opposite quadrants.
	See \cref{fig:gridRamsey} (right).
	Hence, each vertex in $U$ is incident to four edges of color $r$ with endpoints in different quadrants.
\end{proof}

\begin{figure}
	\centering
	\includegraphics{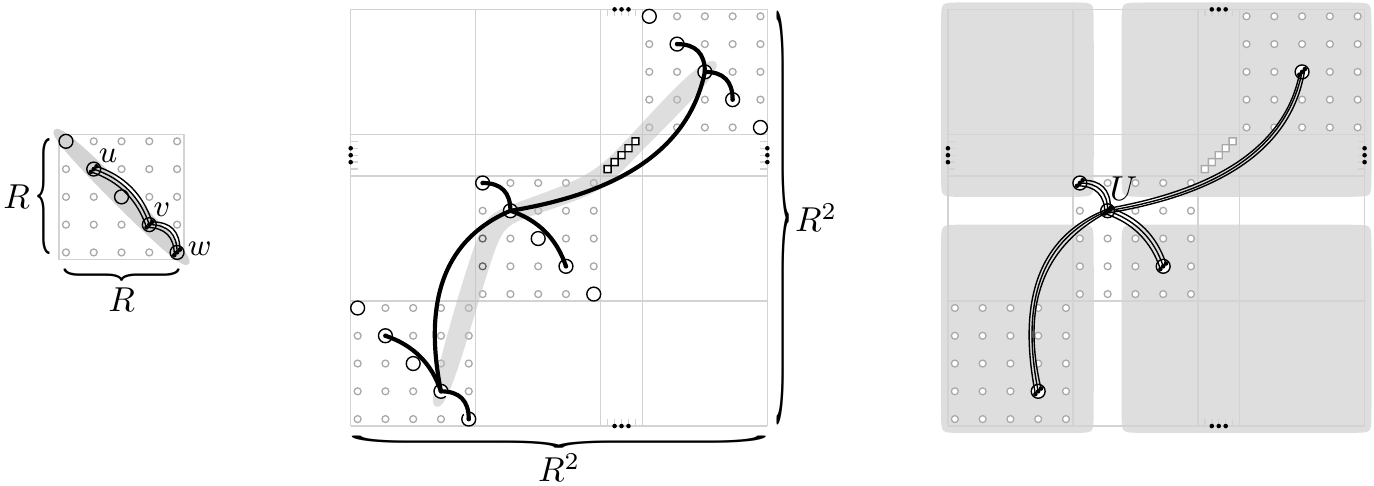}
	\caption{Left: A diagonal in $\Gamma(R,t)$ with monochromatic path $uvw$.
	Middle: In $\Gamma(R^2,t)$ we choose $R$ copies of $\Gamma(R,t)$ (with respective diagonals), arranged in diagonal fashion.
	Right: This leads to a set $U$ of vertices with edges of the same color into each of the four quadrants.}
	\label{fig:gridRamsey}
\end{figure}

We also use the following bound on Erd\H{o}s--Szekeres numbers.

\begin{lemma}[\cite{HMPT20}]\label{lem:ErdSzek}
	There is a constant $c_4$ such that for each positive integer $k$ each set of $2^{k+c_4 \sqrt{k\log k}}$ points in general position in the plane contains a subset of $k$ points in convex position.
\end{lemma}

Finally, we prove that the graph $G(n)$ described in the beginning of this section has geometric thickness at least $3$.

\begin{theorem}\label{thm:main}
	Let $k$, $m$, $n$ be integers with $k\geq c_3$ (with $c_3$ from \cref{lem:gridRamsey} for $r=11$ and $t=5$), $n\geq 2^{c_2 2^{10k^2}}$ (with $c_2$ from \cref{lem:tidyGrid}) and $m\geq 12^n$.
	For each $N \geq 2^{2m+c_4\sqrt{2m \log(2m)}}$ (with $c_4$ from \cref{lem:ErdSzek}) the graph $G(N)$ has geometric thickness at least $3$.
\end{theorem}
\begin{proof}
	Consider any geometric drawing of $G=G(N)$.
	We assume that the vertices are in general position, otherwise we can apply a small perturbation at the vertices to achieve this without introducing any new crossings.
	For the sake of a contradiction, suppose that there is a partition of $G$ into two plane subgraphs $\mathbb{A}$ and $\mathbb{B}$.
	We refer to the sets $\Lambda_0$, $\Lambda_1$, $\Lambda_2$, and $\Lambda_3$ as points sets like in the definition of $G$.
	Our proof proceeds as follows.
	We find a large tidy drawing of $H_n$ with base points in $\Lambda_0$ and subdivision vertices in $\Lambda_1$.
	\cref{lem:tidyGrid} guarantees a dotted grid in this drawing.
	Then we consider the connections of the vertices in the grid cells via $\Lambda_2$.
	We use \cref{lem:gridRamsey} to show that many connections stay within the grid and hence many vertices of $\Lambda_2$ lie in the grid as well.
	Finally, we consider the connections of vertices from $\Lambda_2$ within the grid and use \cref{lem:gridRamsey} again, to find a configuration of vertices from $\Lambda_2$ that leads to a contradiction.

	Consider the point set $\Lambda_0$.
	\Cref{lem:ErdSzek} yields a set $\Lambda_0'\subseteq \Lambda_0$ of $2m$ points in convex position, since $N\geq 2^{2m+c_4\sqrt{2m \log(2m)}}$.
	We consider the points in $\Lambda_0'$ in counterclockwise order with an arbitrary first vertex.
	Consider the copy of $H_m$ in $G$ between the set $A$ of the first $m$ vertices of $\Lambda_0'$ and the set $B$ of the last $m$ vertices of $\Lambda_0'$.
	The edges of the underlying copy of $K_{m,m}$ are of four different types: in $H_m$ they correspond to two edges from $\mathbb{A}$, or to two edges from $\mathbb{B}$, or one edge from $\mathbb{A}$ and one edge from $\mathbb{B}$ (where either the edge from $\mathbb{A}$ has an endpoint in $A$ and the edge from $\mathbb{B}$ has an endpoint in $B$ or vice versa).
	Since $m\geq 12^n$ there is, due to the bipartite Ramsey theorem (stated as \cref{lem:bipartiteRamsey} in \cref*{prf:tidyGrid}), a copy of $K_{n,n}$ with all edges of the same type, leading to a corresponding copy $H$ of $H_n$.
	Since $n\geq 3$, this type cannot be one of the types with edges only from $\mathbb{A}$ or only from $\mathbb{B}$ as both $\mathbb{A}$ and $\mathbb{B}$ are planar but $K_{3,3}$ is not.
	Without loss of generality, assume that all edges in $H$ incident to $A$ are in $\mathbb{A}$ and all edges incident to $B$ are in $\mathbb{B}$.
	Observe that $H$ is a tidy geometric drawing of $H_n$ since $\mathbb{A}$ and $\mathbb{B}$ are crossing-free and the sets $A$ and $B$ are separated (their convex hulls do not intersect and $A\cup B=\Lambda_0'$ is in convex position).
	Further note that $2^{2k^2}\geq (k^2+1)^4 2^8$ for $k\geq 4$.
	Hence $n\geq 2^{c_2 2^{10k^2}} \geq 2^{c_2 (k^2+1)^4 2^{8(k^2+1)}}$, and there is, by \cref{lem:tidyGrid}, a dotted tidy $(k^2+1)$-grid $T$ in $H$ with vertices from $\Lambda_1$ in the cells.

	Let $\Lambda_1'\subseteq \Lambda_1$ denote a set of vertices consisting of one vertex from each cell of $T$.
	Consider the graph $\Gamma_1$ with vertex set $\Lambda_1'$ where two vertices are adjacent if and only if they are in distinct rows and distinct columns of $T$.
	Then $\Gamma_1$ forms a copy of $\Gamma(k^2,1)$.
	We will define an edge coloring $\Phi$ of $\Gamma_1$ based on the drawing of the edges between $\Lambda'_1$ and $\Lambda_2$.
	Consider two vertices $x$, $x'\in \Lambda_1'$.
	There are $89$ vertices in $\Lambda_2$ adjacent to both $x$ and $x'$.
	We will distinguish $11$ different cases how the edges between such $y\in \Lambda_2$ and $x$, $x'$ are drawn.
	Then, by the pigeonhole principle, there will be nine vertices from $\Lambda_2$ with the same type of drawing of $xy$ and $x'y$.
	The cases are not disjoint from each other and we break ties arbitrarily.
	If there are nine vertices $y\in \Lambda_2$ with $xy$, $x'y\in E(\mathbb{A})$, then $\Phi(xx')=1$.
	If there are nine such vertices with $xy$, $x'y\in E(\mathbb{B})$, then $\Phi(xx')=2$.
	Now assume that there are no such nine vertices, so there are $73$ such vertices where one edge is from $\mathbb{A}$ and the other edge is from $\mathbb{B}$.
	These edges either leave $T$ or stay within $T$.
	If we have at least nine vertices that stay within $T$, we pick $\Phi(xx')=11$.
	Otherwise, we can assume that there are at least 65 vertices $y$, for which the bicolored path $xyx'$ leaves $T$.
	The cell containing $x$ is the intersection of an $A$-corridor and a $B$-corridor of $T$.
	So an edge $xy$ intersects the boundary of $T$ either at one of the two \enquote{ends} of the $A$-corridor (if $xy\in E(\mathbb{A})$) or at one of the two \enquote{ends} of the $B$-corridor (if $xy\in E(\mathbb{B})$).
	Similarly, an edge $x'y$ has four options to leave $T$.
	Also observe that each of $xy$ and $x'y$ can intersect the boundary of $T$ only once, see \cref{fig:nonStretchable}.
	The figure shows the boundary edges of $T$ and a supposedly straight-line segment $L'$ intersecting the boundary twice.
	This arrangement can't be realized by straight lines as the straight line through $L'$ intersects itself once or some other line twice otherwise.
	This gives $8$ possibilities how the intersections can be located (under the assumption that $xy$ and $x'y$ are not both in $\mathbb{A}$ and not both in $\mathbb{B}$).
	We use colors $3,\ldots,10$ to encode these possibilities.
	Whenever there is a set $\hat{Y}$ of nine vertices from $\Lambda_2$ such that the paths $xyx'$ have the same locations of intersections for all $y\in\hat{Y}$, the edge $xx'$ receives the corresponding color.
	If $xx'$ is neither colored with $1$,~$2$, or $11$, we have at least 65 vertices connected via leaving $T$, and therefore at least one of the eight possibilities how to leave $T$ occurs nine times.
	So $\Phi$ is well defined (up to breaking ties arbitrarily).

	\begin{figure}
		\centering
		\includegraphics{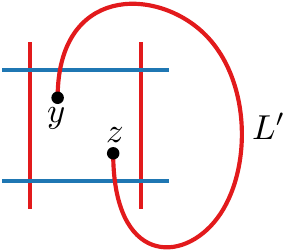}
		\caption{This arrangement is not realizable by straight-line segments, since the straight line through $L'$ does not intersect any of the other lines twice and does not intersect itself.}
		\label{fig:nonStretchable}
	\end{figure}

	We claim that $\Phi$ is admissible.
	We first prove that colors $3,\ldots,10$ do not induce a monotone monochromatic path on two edges.
	For the sake of a contradiction, suppose that there is such a path $xx'x''$.
	By symmetry, we assume that there are vertices $y$, $y'$ and edges $xy'$, $x'y\in E(\mathbb{A})$, and $x'y'$, $x''y\in E(\mathbb{B})$ such that $xy'$ and $x'y$ leave $T$ at the same sides of their respective $A$-corridors and $x'y'$ and $x''y$ leave $T$ at the same sides of their respective $B$-corridors.
	The situation is depicted in \cref{fig:noMonMonPath}.
	We claim that this arrangement is not stretchable.
	To see this consider the 4-cycle between the intersections of $xy', x''y$ and the grid boundary as depicted in~\cref{fig:noMonMonPath} (right).
	This cycle needs to be embedded as a quadrilateral.
	For two opposing corners (the depicted crossings $L_1$/$L_2$ and $L_x$/$L_{x''}$) we have to embed the edges such that the \enquote{stubs} lie in the inside of the quadrilateral.
	To achieve this for one corner we need an incident concave angle in the quadrilateral and hence the realization of the quadrilateral would require at least two concave angles, which is not possible.
	Hence, such an arrangement is not stretchable.
	As a consequence, the colors $3,\ldots,10$ do not induce a monotone monochromatic path on two edges.
	This immediately shows that these colors also do not induce a monochromatic copy of $K_5$.
	The color classes $1$ and $2$ correspond to subgraphs of the plane graphs $\mathbb{A}$ and $\mathbb{B}$, respectively.
	Hence, they do not induce monochromatic copies of $K_5$ as well.
	This shows that all monochromatic copies of $K_5$ are of color $r=11$.
	Therefore, $\Phi$ is admissible.

	\begin{figure}
		\centering
		\includegraphics{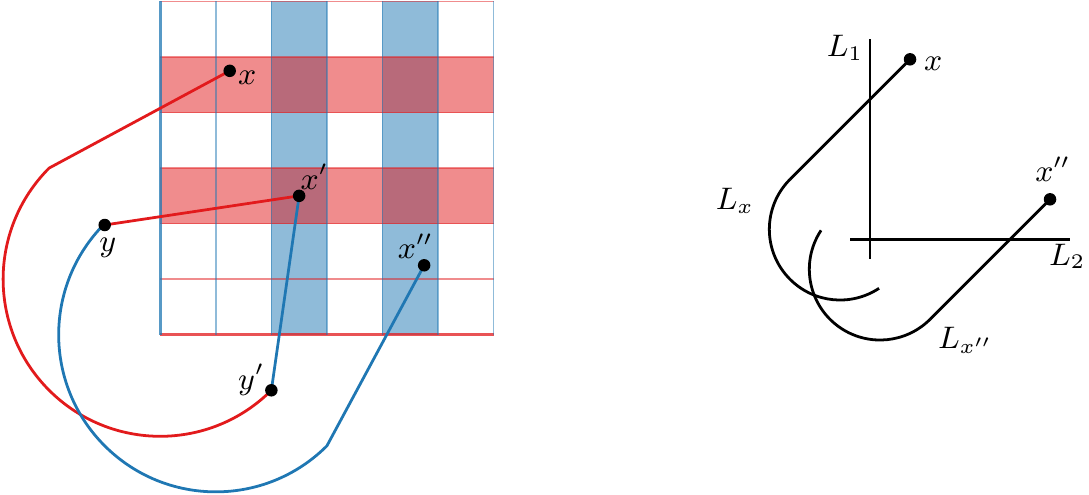}
		\caption{Left: A monotone path that is monochromatic under $\Phi$ in some color in $\{3,\ldots,10\}$.
		Note that it is not possible that $x'y$ and $x'y'$ intersect.
		Right: The edges from the left part forming an arrangement that can't be realized by straight-line segments.}
		\label{fig:noMonMonPath}
	\end{figure}

	Now divide the $(k^2+1)$-grid $T$ into $k^2$ many $(k+1)$-grids $T_i^j$, with $i$, $j,\le k$, where $T_i^j$ consists of the $A$-edges on position $(i-1)k+1,\ldots,ik+1$ (in the ordering of $A$) and the $B$-edges with positions $(j-1)k+1,\ldots,jk+1$ (in the ordering of $B$).
	See \cref{fig:splitGrid}.
	Let $\Gamma_i^j$ denote the subgraph of $\Gamma_1$ corresponding to $T_i^j$.
	Then $\Gamma_i^j$ is a copy of $\Gamma(k,1)$ and $\Phi$ is an admissible $11$-coloring of $\Gamma_i^j$.
	Consider some fixed $i$, $j\le k$.
	Due to the choice of $k$ there is, by \cref{lem:gridRamsey}, an edge $xx'$ in $\Gamma_i^j$ of color $r=11$ (we do not need the stronger statement of the \lcnamecref{lem:gridRamsey} here).
	Hence, there is a set $Y_i^j\in \Lambda_2$ of nine vertices such that for each $y\in Y_i^j$ the edges $xy$ and $x'y$ stay within $T$.
	Let $\mathcal{A}_x$ and $\mathcal{B}_x$ denote the $A$-corridor and $B$-corridor whose intersection forms the cell containing $x$.
	Similarly, let $\mathcal{A}_{x'}$ and $\mathcal{B}_{x'}$ denote the respective corridors for the cell containing $x'$.
	As argued above, edges within $T$ cannot leave their respective corridors.
	So each $y\in Y_i^j$ lies either in the cell $\mathcal{A}_x\cap \mathcal{B}_{x'}$ or in the cell $\mathcal{B}_x\cap \mathcal{A}_{x'}$.
	By the pigeonhole principle, there is a set $\tilde{Y}_i^j\subseteq Y_i^j$ of five vertices that lie in the same cell of $T$.
	Note that this cell is contained in $T_i^j$.

	\begin{figure}
		\centering
		\includegraphics{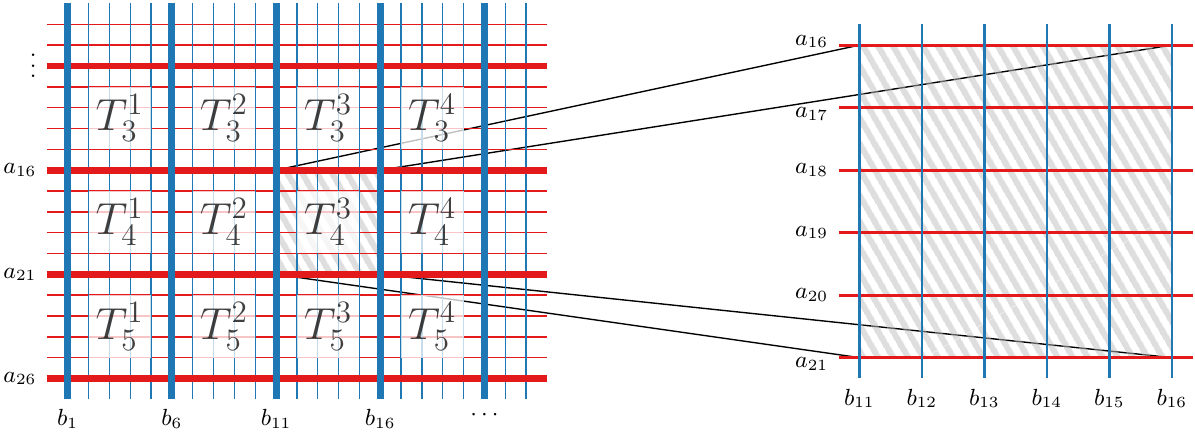}
		\caption[A k-squared-plus-one grid T contains k squared many k-plus-one subgrids.]
		{A $(k^2+1)$-grid $T$ contains $k^2$ many $(k+1)$-subgrids.
		(Here $k = 5$.)}
		\label{fig:splitGrid}
	\end{figure}

	Consider the copy of $\Gamma(k,5)$ whose vertex set consists of the union of all sets $\tilde{Y}_i^j$, with $i$, $j\le k$, where two vertices $y\in \tilde{Y}_i^j$ and $y'\in \tilde{Y}_{i'}^{j'}$ are connected if and only if $i\neq i'$ and $j\neq j'$.
	For any two vertices $y$, $y'\in V(\Gamma(k,5))$ there is a (unique) vertex in $\Lambda_3$ adjacent to both vertices.
	We define a coloring $\Psi$ of the edges of $\Gamma(k,5)$ similar to the coloring $\Phi$ above, except that the color of an edge $yy'$ in $\Gamma(k,5)$ is determined by the drawing of the unique edges $yz$ and $y'z$, $z\in \Lambda_3$ (instead of a set of nine edge pairs behaving identically).
	Then $\Psi$ is admissible by arguments similar to those applied for $\Phi$.
	Due to the choice of $k$ there are, by \cref{lem:gridRamsey}, indices $i$, $j\le k$ such that each vertex in $\tilde{Y}_i^j$ is incident to four edges of color $11$ under $\Psi$ with endpoints in different $(i,j)$-quadrants of $\Gamma(k,5)$.
	The situation is depicted in \cref{fig:fourDirections}.

	\begin{figure}
		\centering
		\includegraphics{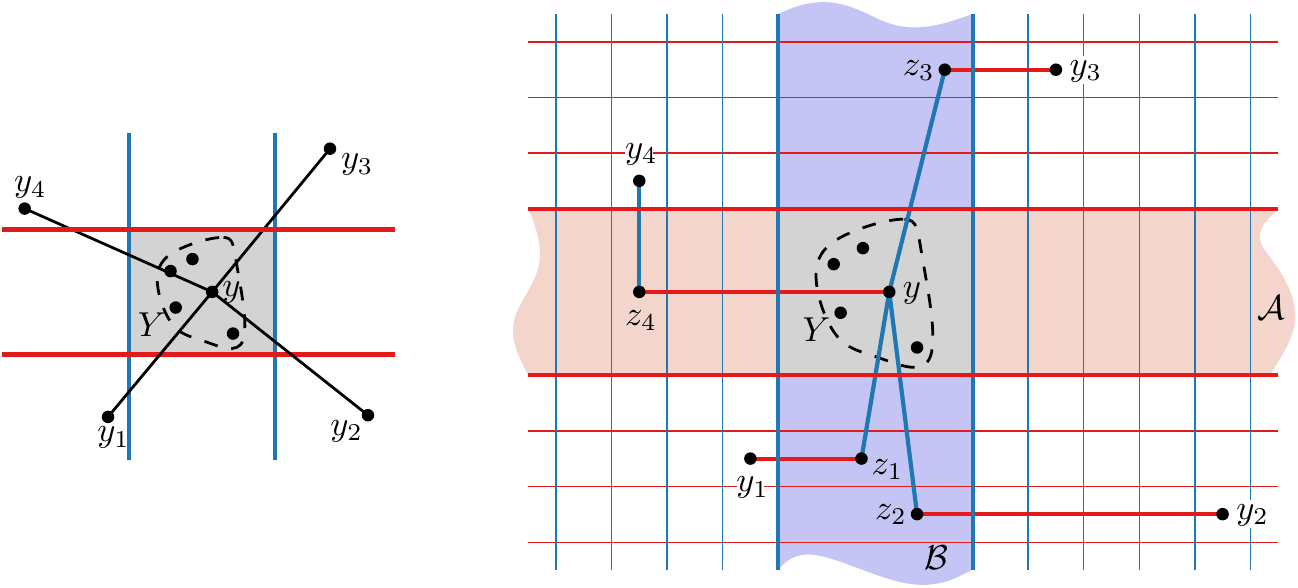}
		\caption{Left: Every vertex $y\in Y$ is incident to four edges in $\Gamma(k,5)$ of color $11$ with endpoints in different quadrants.
		Right: In $G$, each $y\in Y$ has two edges of the same type ($\mathbb{A}$/$\mathbb{B}$) that leave in the same direction relative to $y$; here $yz_1$ and $yz_2$.}
		\label{fig:fourDirections}
	\end{figure}

	Let $Y=\tilde{Y}_i^j$ for the specific indices $i$ and $j$ from above.
	Consider the $A$-corridor $\mathcal{A}$ and $B$-corridor $\mathcal{B}$ of $T$ whose intersection forms the cell containing the set $Y$.
	For a vertex $y\in Y$ consider four vertices $y_1,\ldots,y_4$ from different quadrants with $\Psi(yy_\ell)=11$, $\ell=1,\ldots,4$.
	Each edge $yy_\ell \in \Gamma(k,5)$ corresponds to two edges (of $G$) $yz_\ell$ and $y_\ell z_\ell$ for some $z_\ell\in \Lambda_3$ such that $z_\ell$ lies within $T$.
	In particular, $z_\ell$ lies either in $\mathcal{A}$ or in $\mathcal{B}$ but not in the cell containing $y$.
	As $y_1,\ldots,y_4$ are from four different quadrants, two of the vertices $z_1,\ldots,z_4$ lie in $\mathcal{A}$ or two lie in $\mathcal{B}$.
	Moreover, for either $\mathcal{A}$ or $\mathcal{B}$ two vertices lie on different \enquote{sides} of $y$ within the corridor.
	If for $y$ we have $\lvert\mathcal{A}\cap \{z_1,\ldots,z_4\}\rvert \geq 2$ and at least two of these vertices lie on different sides in $\mathcal{A}$ relative to $y$, we call $y$ an $\mathcal{A}$-vertex, otherwise we call $y$ a $\mathcal{B}$-vertex.

	To get a contradiction we now show that $Y$ contains at most two $\mathcal{A}$-vertices and at most two $\mathcal{B}$-vertices, which violates $\lvert Y\rvert=5$.
	Due to the choice of $Y\subseteq Y_i^j$, there are vertices $x$, $x'\in V(T_i^j)=V(\Gamma_i^j)$ such that there are edges $xy\in E(\mathbb{A})$ and $x'y\in E(\mathbb{B})$ with $\Phi(xy)=\Phi(x'y)=11$.
	That is, $xy\in \mathcal{A}$ and $x'y\in\mathcal{B}$.
	For the sake of a contradiction, suppose that there are three $\mathcal{A}$-vertices $y$, $y'$, $y''$ in $Y$.
	Then there are three vertices $\tilde{y}$, $\tilde{y}'$, $\tilde{y}''\in V(\Gamma(k,5))\subseteq \Lambda_2$ and three vertices $z$, $z'$, $z''\in\Lambda_3$ such that $yz$, $y'z'$, $y''z''\in E(\mathbb{A})$, $\tilde{y}z$, $\tilde{y}'z'$, $\tilde{y}''z''\in E(\mathbb{B})$, and $z$, $z'$, $z''$ lie in $\mathcal{A}$ on the same side relative to $y$, but not in $T_i^j$.
	By the same reasoning we can find three vertices $\tilde{z},\tilde{z}',\tilde{z}''$ such that $y\tilde{z},y'\tilde{z}',y''\tilde{z}''\in E(\mathbb{A})$, but now these vertices lie on the other side in $\mathcal{A}$ relative to $x$ (but also outside $T_i^j$).
	The edges $L=\{y\tilde{z},yz,y'\tilde{z}',y'z',y''\tilde{z}'',y''z''\}$ split $T_i^j$ in four zones.
	In one of these zones, $x$ has to be located.
	No matter which zone we pick, there will always be a crossing of an edge from $\{xy,xy',xy''\}\subseteq E(\mathbb{A})$ with an edge in $L\subseteq E(\mathbb{A})$ (see~\cref{fig:mainargument}), a contradiction.
	Consequently, there are no three $\mathcal{A}$-vertices in $Y$.

	\begin{figure}
		\centering
		\includegraphics{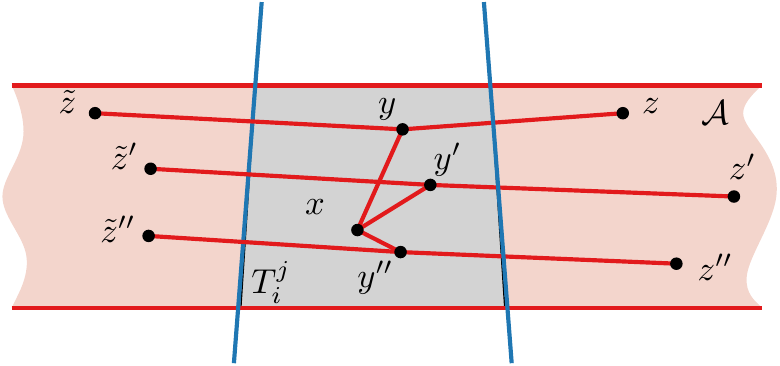}
		\caption{Construction in the proof of~\cref{thm:main}.
		Obtaining a monochromatic crossing at $xy$, $xy'$ or $xy''$ is unavoidable.}
		\label{fig:mainargument}
	\end{figure}

	Similarly, there are no three $\mathcal{B}$-vertices in $Y$.
	This contradicts $\lvert Y\rvert \geq 5$.
	Hence, the geometric thickness of $G$ is at least $3$.
\end{proof}

\Cref{thm:lowerBound} is a direct consequence of~\cref{thm:main}.

\clearpage


\section{Conclusions}

We proved that the largest geometric thickness among $2$-degenerate graphs is either $3$ or $4$, answering two questions posed by Eppstein~\cite{EppGeomThickness}.
It remains open to decide whether there is a $2$-degenerate graph of geometric thickness or geometric arboricity $4$.

Our proof of the lower bound shows a geometric thickness of at least $3$ for a tremendously large $2$-degenerate graph.
This is mainly due to using several rounds of Ramsey type arguments.
We make little attempts to reduce this size and there are several places in the proof where a smaller size could be attained easily, for instance by using better or more specific Ramsey numbers (\cref{lem:bipartiteRamsey}, \cref{lem:gridRamsey}).
In one step in the proof (\cref{lem:kgrid}) we are given a collection of red and blue straight-line segments in the plane and we need to find $k$ red segments and $k$ blue segments forming a grid combinatorially equivalent to $G_k$.
We need exponentially many segments to be given, however it seems that a linear number suffices.
An arrangement of $3k$ red segments and $3k$ blue segments without copy of $G_{k+1}$ is given in \cref{fig:noGk}.

\begin{figure}
	\centering
	\includegraphics{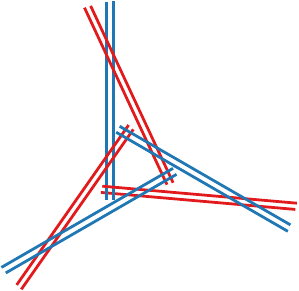}
	\caption{This arrangement of $3k$ red segments and $3k$ blue segments contains no copy of $G_{k+1}$.
	For each color and each slope there are $k$ parallel segments (here $k=2$ is depicted).}
	\label{fig:noGk}
\end{figure}

\begin{question}
	Given an arrangement of $3k$ disjoint red straight-line segments and $3k$ disjoint blue straight-line segments, where each red segment crosses each blue segment, are there always $k$ red segments and $k$ blue segments forming a grid combinatorially equivalent to $G_k$?
\end{question}

The $2$-degenerate graphs form a subclass of Laman graphs, which in turn form a subclass of all graphs of arboricity $2$.
Our lower bound gives a graph of geometric thickness $3$ in either of these classes.
However, for both larger classes it is unknown whether the geometric thickness is bounded by a constant from above.

\bibliographystyle{plainurl}
\bibliography{refs}

\end{document}